\documentclass[12pt, a4paper, reqno]{amsart}
\usepackage{amsmath}
\usepackage{amsfonts}
\usepackage{amssymb}
\usepackage{color}
\usepackage{comment,graphicx}
\usepackage[T1]{fontenc}
\usepackage{url}
\usepackage{ae}
\usepackage{mathtools}
\usepackage{esint}
\usepackage{latexsym}
\usepackage{amscd}

\def\phi{\varphi}
\def\rho{\varrho}
\def\epsilon{\varepsilon}
\numberwithin{equation}{section}
\theoremstyle{plain}
\newtheorem{theorem}[equation]{Theorem}
\newtheorem{lemma}[equation]{Lemma}
\newtheorem{proposition}[equation]{Proposition}
\newtheorem{corollary}[equation]{Corollary}
\theoremstyle{definition}
\newtheorem{definition}[equation]{Definition}

\theoremstyle{remark}
\newtheorem{remark}[equation]{Remark}
\newtheorem{example}[equation]{Example}

\renewcommand{\leq}{\leqslant}
\renewcommand{\geq}{\geqslant}

\begin{document}
\title[Weigthed Triebel-Lizorkin spaces]{Duality of Triebel-Lizorkin spaces
of general weights}
\author[D. Drihem]{Douadi Drihem}
\address{Douadi Drihem\\
 Laboratory of Functional Analysis and Geometry of Spaces\\
Faculty of Mathematics and Informatics\\
Department of Mathematics\\
M'sila University\\
PO Box 166 Ichebelia, 28000 M'sila, Algeria.}
\email{douadidr@yahoo.fr, douadi.drihem@univ-msila.dz}
\thanks{ }
\date{\today }

\begin{abstract}
In this paper, we identify the duals of Triebel-Lizorkin spaces of
generalized smoothness. In some particular cases these function spaces are
just weighted Triebel-Lizorkin spaces. To do these, we will be working at
the level of sequence spaces. The $\varphi $-transform characterization of
these function spaces in the sense of Frazier and Jawerth, and new weighted
version of vector-valued maximal inequality of Fefferman and Stein are the
main tools.
\end{abstract}

\maketitle

\section{Introduction}

It is well-known that duality is an important concept when we study function
spaces. It applied to real interpolation and embeddings. For classical
function spaces such as Lebesgue spaces, Sobolev spaces,\ Hardy spaces,
Besov and Triebel-Lizorkin spaces are known, see for example \cite{M08}, 
\cite[2.11]{T1} and references therein.

In this direction, the paper attempts to present the duality of
Triebel-Lizorkin spaces of generalized smoothness, were introduced and
studied in \cite{D20} and \cite{D20.1}, which are defined as follows. Select
a Schwartz function $\varphi $ such that 
\begin{equation*}
\text{supp}\mathcal{F}\varphi \subset \big\{\xi :\frac{1}{2}\leq |\xi |\leq 2%
\big\},\quad |\mathcal{F}\varphi (\xi )|\geq c\text{\quad if\quad }\frac{3}{5%
}\leq |\xi |\leq \frac{5}{3}
\end{equation*}%
where $c>0$. Here $\mathcal{F}(\varphi )$ denotes the Fourier transform of $%
\varphi $, defined by%
\begin{equation*}
\mathcal{F}\varphi (\xi ):=(2\pi )^{-n/2}\int_{\mathbb{R}^{n}}e^{-ix\cdot
\xi }\varphi (x)dx,\quad \xi \in \mathbb{R}^{n}.
\end{equation*}

Let%
\begin{equation*}
\mathcal{S}_{\infty }(\mathbb{R}^{n}):=\Big\{\varphi \in \mathcal{S}(\mathbb{%
R}^{n}):\int_{\mathbb{R}^{n}}x^{\beta }\varphi (x)dx=0\text{ for all
multi-indices }\beta \in \mathbb{N}_{0}^{n}\Big\}.
\end{equation*}%
Following Triebel \cite{T1}, we consider $\mathcal{S}_{\infty }(\mathbb{R}%
^{n})$ as a subspace of $\mathcal{S}(\mathbb{R}^{n})$, including the
topology. Let $\mathcal{S}_{\infty }^{\prime }(\mathbb{R}^{n})$ be the
topological dual of $\mathcal{S}_{\infty }(\mathbb{R}^{n})$ and we put $%
\varphi _{k}=2^{kn}\varphi (2^{k}\cdot )$, $k\in \mathbb{Z}$.\textrm{\ }Let $%
0<p<\infty $ and $0<q\leq \infty $. Let $\{t_{k}\}$ be a $p$-admissible
sequence i.e., $t_{k}\in L_{p}^{\mathrm{loc}}(\mathbb{R}^{n})$, $k\in 
\mathbb{Z}$. The Triebel-Lizorkin space\ $\dot{F}_{p,q}(\mathbb{R}%
^{n},\{t_{k}\})$\ is the collection of all $f\in \mathcal{S}_{\infty
}^{\prime }(\mathbb{R}^{n})$\ such that 
\begin{equation*}
\big\|f|\dot{F}_{p,q}(\mathbb{R}^{n},\{t_{k}\})\big\|:=\Big\|\Big(%
\sum\limits_{k=-\infty }^{\infty }t_{k}^{q}|\varphi _{k}\ast f|^{q}\Big)^{%
\frac{1}{q}}|L_{p}(\mathbb{R}^{n})\Big\|<\infty
\end{equation*}%
with the usual modifications if $q=\infty $.

The function spaces $\dot{F}_{p,q}(\mathbb{R}^{n},\{t_{k}\})$ are based on
the weighted class of Tyulenev\ in \cite{Ty15} and \cite{Ty-151} which
introduced a new family of Besov spaces of variable smoothness.

The study of these type of function spaces can be traced back to the late
60s and early 70s. We refer, for instance, to Bownik \cite{M07}, Cobos and
Fernandez \cite{CF88}, Goldman \cite{Go79} and \cite{Go83}, and Kalyabin 
\cite{Ka83}, see also\ Besov \cite{B03} and \cite{B05}, and Kalyabin and
Lizorkin \cite{Kl87}.

The theory of these spaces had a remarkable development in part due to its
usefulness in applications. For instance, they appear in the study of trace
spaces on fractals, see Edmunds and Triebel \cite{ET96} and \cite{ET99},
were they introduced the spaces $B_{p,q}^{s,\Psi }$, where $\Psi $ is a
so-called admissible function, typically of log-type near $0$. For a
complete treatment of these spaces we refer the reader the work of Moura 
\cite{Mo01}. More general function spaces of generalized smoothness can be
found in Farkas and Leopold \cite{FL06}, and reference therein.

Dominguez and Tikhonov in \cite{DT} gave a treatment of function spaces with
logarithmic smoothness (Besov, Sobolev, Triebel-Lizorkin), including various
new characterizations for Besov norms in terms of different, sharp estimates
for Besov norms of derivatives and potential operators (Riesz and Bessel
potentials) in terms of norms of functions themselves and sharp embeddings
between the Besov spaces defined by differences and by Fourier-analytical
decompositions as well as between Besov and Sobolev/Triebel-Lizorkin spaces.

The paper is organized as follows. First we give some preliminaries and
recall some basic facts on the Muckenhoupt classes and the weighted class of
Tyulenev. Also we give some key technical lemmas needed in the proofs of the
main statements. Especially, the weighted version of vector-valued maximal
inequality of Fefferman and Stein. In Section 2, we present some properties
of $\dot{F}_{p,q}(\mathbb{R}^{n},\{t_{k}\})$ spaces. In addition, we
introduce new class of function spaces $\dot{F}_{\infty ,q}(\mathbb{R}%
^{n},\{t_{k}\})$ which unify and generalize the Triebel-Lizorkin spaces $%
\dot{F}_{\infty ,q}(\mathbb{R}^{n},\{2^{ks}\})$. In Section 3, we identify
the duals of $\dot{F}_{p,q}(\mathbb{R}^{n},\{t_{k}\})$ spaces and the most
interesting case is when $p=1.$

\section{Background tools}

In this section we recall some basic facts on Muckenhoupt class and the
weight class $\dot{X}_{\alpha ,\sigma ,p}$. First we make some notation and
conventions.

\subsection{Notation and conventions}

Throughout this paper, we denote by $\mathbb{R}^{n}$ the $n$-dimensional
real Euclidean space, $\mathbb{N}$ the collection of all natural numbers and 
$\mathbb{N}_{0}=\mathbb{N}\cup \{0\}$. The letter $\mathbb{Z}$ stands for
the set of all integer numbers.\ The expression $f\lesssim g$ means that $%
f\leq c\,g$ for some independent constant $c$ (and non-negative functions $f$
and $g$), and $f\approx g$ means $f\lesssim g\lesssim f$. \vskip5pt

For $x\in \mathbb{R}^{n}$ and $r>0$ we denote by $B(x,r)$ the open ball in $%
\mathbb{R}^{n}$ with center $x$ and radius $r$. By supp $f$ we denote the
support of the function $f$, i.e., the closure of its non-zero set. If $%
E\subset {\mathbb{R}^{n}}$ is a measurable set, then $|E|$ stands for the
(Lebesgue) measure of $E$ and $\chi _{E}$ denotes its characteristic
function. By $c$ we denote generic positive constants, which may have
different values at different occurrences. \vskip5pt

A weight is a nonnegative locally integrable function on $\mathbb{R}^{n}$
that takes values in $(0,\infty )$ almost everywhere. For measurable set $%
E\subset \mathbb{R}^{n}$ and a weight $\gamma $, $\gamma (E)$ denotes 
\begin{equation*}
\int_{E}\gamma (x)dx.
\end{equation*}%
Given a measurable set $E\subset \mathbb{R}^{n}$ and $0<p\leq \infty $, we
denote by $L_{p}(E)$ the space of all functions $f:E\rightarrow \mathbb{C}$
equipped with the quasi-norm 
\begin{equation*}
{\big\|}f|L_{p}(E){\big\|}:=\Big(\int_{E}\left\vert f(x)\right\vert ^{p}dx%
\Big)^{1/p}<\infty ,
\end{equation*}%
with $0<p<\infty $ and%
\begin{equation*}
\left\Vert f|L_{\infty }(E)\right\Vert :=\underset{x\in E}{\text{ess-sup}}%
\left\vert f(x)\right\vert <\infty .
\end{equation*}%
For a function $f$ in $L_{1}^{\mathrm{loc}}$, we set 
\begin{equation*}
M_{A}(f):=\frac{1}{|A|}\int_{A}\left\vert f(x)\right\vert dx
\end{equation*}%
for any $A\subset \mathbb{R}^{n}$. Furthermore, we put%
\begin{equation*}
M_{A,p}(f):=\Big(\frac{1}{|A|}\int_{A}\left\vert f(x)\right\vert ^{p}dx\Big)%
^{\frac{1}{p}},
\end{equation*}%
with $0<p<\infty $. Further, given a measurable set $E\subset \mathbb{R}^{n}$
and a weight $\gamma $, we denote the space of all functions $f:\mathbb{R}%
^{n}\rightarrow \mathbb{C}$ with finite quasi-norm 
\begin{equation*}
{\big\|}f|L_{p}(\mathbb{R}^{n},\gamma ){\big\|}={\big\|}f\gamma |L_{p}(%
\mathbb{R}^{n}){\big\|}
\end{equation*}%
by $L_{p}(\mathbb{R}^{n},\gamma )$.

If $1\leq p\leq \infty $ and $\frac{1}{p}+\frac{1}{p^{\prime }}=1$, then $%
p^{\prime }$ is called the conjugate exponent of $p$. Let $0<p,q\leq \infty $%
. The space $L_{p}(\ell _{q})$ is defined to be the set of all sequences $%
\{f_{k}\}$ of functions such that%
\begin{equation*}
\big\|\{f_{k}\}|L_{p}(\ell _{q})\big\|:=\Big\|\Big(\sum_{k=-\infty }^{\infty
}|f_{k}|^{q}\Big)^{\frac{1}{q}}|L_{p}(\mathbb{R}^{n})\Big\|<\infty
\end{equation*}%
with the usual modifications if $q=\infty $ and if $\{t_{k}\}$ is a sequence
of functions then%
\begin{equation*}
\big\|\{f_{k}\}|L_{p}(\ell _{q},\{t_{k}\})\big\|=\big\|\{t_{k}f_{k}\}|L_{p}(%
\ell _{q})\big\|.
\end{equation*}

The symbol $\mathcal{S}(\mathbb{R}^{n})$ is used in place of the set of all
Schwartz functions on $\mathbb{R}^{n}$. In what follows, $Q$\ will denote an
cube in the space $\mathbb{R}^{n}$\ with sides parallel to the coordinate
axes and $l(Q)$\ will denote the side length of the cube $Q$. For $k\in 
\mathbb{Z}$ and $m\in \mathbb{Z}^{n}$, denote by $Q_{k,m}$ the dyadic cube $%
Q_{k,m}:=2^{-k}([0,1)^{n}+m)$. For the collection of all such cubes we use $%
\mathcal{Q}:=\{Q_{k,m}:k\in \mathbb{Z},m\in \mathbb{Z}^{n}\}$.

\subsection{Muckenhoupt weights}

The purpose of this subsection is to review some known properties of\
Muckenhoupt class.

\begin{definition}
Let $1<p<\infty $. We say that a weight $\gamma $ belongs to the Muckenhoupt
class $A_{p}(\mathbb{R}^{n})$ if there exists a constant $C>0$ such that for
every cube $Q$ the following inequality holds 
\begin{equation}
M_{Q}(\gamma )M_{Q,\frac{p^{\prime }}{p}}(\gamma ^{-1})\leq C.
\label{Ap-constant}
\end{equation}
\end{definition}

The smallest constant $C$ for which $\mathrm{\eqref{Ap-constant}}$ holds,
denoted by $A_{p}(\gamma )$. As an example, we can take $\gamma
(x)=|x|^{\alpha },\alpha \in \mathbb{R}$. Then $\gamma \in A_{p}(\mathbb{R}%
^{n})$, $1<p<\infty $, if and only if $-n<\alpha <n(p-1)$.

For $p=1$ we rewrite the above definition in the following way.

\begin{definition}
We say that a weight $\gamma $ belongs to the Muckenhoupt class $A_{1}(%
\mathbb{R}^{n})$ if there exists a constant $C>0$ such that for every cube $%
Q $\ and for a.e.\ $y\in Q$ the following inequality holds 
\begin{equation}
M_{Q}(\gamma )\leq C\gamma (y).  \label{A1-constant}
\end{equation}
\end{definition}

The smallest constant $C$ for which $\mathrm{\eqref{A1-constant}}$ holds,
denoted by $A_{1}(\gamma )$. The above classes have been first studied by
Muckenhoupt\ \cite{Mu72} and use to characterize the boundedness of the
Hardy-Littlewood maximal function on $L^{p}(\gamma )$, see the monographs 
\cite{GR85} and \cite{L. Graf14}\ for a complete account on the theory of
Muckenhoupt weights.

We recall a few basic properties of the class $A_{p}(\mathbb{R}^{n})$
weights, see \cite[Chapter 7]{L. Graf14} and \cite[Chapter 5]{St93}.

\begin{lemma}
\label{Ap-Property}Let $1\leq p<\infty $.\newline
$\mathrm{(i)}$ Let $1<p<\infty $. $\gamma \in A_{p}(\mathbb{R}^{n})$ if and
only if $\gamma ^{1-p^{\prime }}\in A_{p^{\prime }}(\mathbb{R}^{n})$.\newline
$\mathrm{(ii)}$ Let $\gamma \in A_{p}(\mathbb{R}^{n})$. There is $C>0$ such
that for any cube $Q$ and a measurable subset $E\subset Q$%
\begin{equation*}
\Big(\frac{|E|}{|Q|}\Big)^{p-1}M_{Q}(\gamma )\leq CM_{E}(\gamma ).
\end{equation*}%
$\mathrm{(iii)}$ Let $1\leq p<\infty $ and $\gamma \in A_{p}(\mathbb{R}^{n})$%
. Then there exist $\delta \in (0,1)$ and $C>0$ depending only on $n$, $p$,
and $A_{p}(\gamma )$ such that for any cube $Q$ and any measurable subset $S$
of $Q$ we have%
\begin{equation*}
\frac{M_{S}(\gamma )}{M_{Q}(\gamma )}\leq C\Big(\frac{|S|}{|Q|}\Big)^{\delta
-1}.
\end{equation*}%
$\mathrm{(iv)}$ Suppose that $\gamma \in A_{p}(\mathbb{R}^{n})$ for some $%
1<p<\infty $. Then there exists a $1<p_{1}<p<\infty $ such that $\gamma \in
A_{p_{1}}(\mathbb{R}^{n})$.
\end{lemma}

\subsection{The weight class $\dot{X}_{\protect\alpha ,\protect\sigma ,p}$}

Let $0<p\leq \infty $. A weight sequence $\{t_{k}\}$ is called $p$%
-admissible if $t_{k}\in L_{p}^{\mathrm{loc}}(\mathbb{R}^{n})$ for all $k\in 
\mathbb{Z}$. We mention here that 
\begin{equation*}
\int_{E}t_{k}^{p}(x)dx<c(k)
\end{equation*}%
for any $k\in \mathbb{Z}$ and any compact set $E\subset \mathbb{R}^{n}$. For
a $p$-admissible weight sequence $\{t_{k}\}$\ we set%
\begin{equation*}
t_{k,m,p}:=\big\|t_{k}|L_{p}(Q_{k,m})\big\|,\quad k\in \mathbb{N}_{0},m\in 
\mathbb{Z}^{n}.
\end{equation*}

Tyulenev\ in \cite{Ty15} and \cite{Ty-151} introduced the following new
weighted class\ and use it to study Besov spaces of variable smoothness.

\begin{definition}
\label{Tyulenev-class}Let $\alpha _{1}$, $\alpha _{2}\in \mathbb{R}$, $%
p,\sigma _{1}$, $\sigma _{2}$ $\in (0,+\infty ]$, $\alpha =(\alpha
_{1},\alpha _{2})$ and let $\sigma =(\sigma _{1},\sigma _{2})$. We let $\dot{%
X}_{\alpha ,\sigma ,p}=\dot{X}_{\alpha ,\sigma ,p}(\mathbb{R}^{n})$ denote
the set of $p$-admissible weight sequences $\{t_{k}\}$ satisfying the
following conditions. There exist numbers $C_{1},C_{2}>0$ such that for any $%
k\leq j$\ and every cube $Q,$%
\begin{eqnarray}
M_{Q,p}(t_{k})M_{Q,\sigma _{1}}(t_{j}^{-1}) &\leq &C_{1}2^{\alpha _{1}(k-j)},
\label{Asum1} \\
M_{Q,p}^{-1}(t_{k})M_{Q,\sigma _{2}}(t_{j}) &\leq &C_{2}2^{\alpha _{2}(j-k)}.
\label{Asum2}
\end{eqnarray}
\end{definition}

The constants $C_{1},C_{2}>0$ are independent of both the indexes $k$ and $j$%
.

\begin{remark}
$\mathrm{(i)}$\ We would like to mention that if $\{t_{k}\}$ satisfying $%
\mathrm{\eqref{Asum1}}$ with $\sigma _{1}=r\left( \frac{p}{r}\right)
^{\prime }$ and $0<r<p\leq \infty $, then $t_{k}^{p}\in A_{\frac{p}{r}}(%
\mathbb{R}^{n})$ for any $k\in \mathbb{Z}$.\newline
$\mathrm{(ii)}$ We say that $t_{k}\in A_{p}(\mathbb{R}^{n})$,\ $k\in \mathbb{%
Z}$, $1<p<\infty $ have the same Muckenhoupt constant if%
\begin{equation*}
A_{p}(t_{k})=c,\quad k\in \mathbb{Z},
\end{equation*}%
where $c$ is independent of $k$.\newline
$\mathrm{(iii)}$ Definition \ref{Tyulenev-class} is different from the one
used in \cite[Definition 2.7]{Ty15}, because we used the boundedness of the
maximal function on weighted Lebesgue spaces.
\end{remark}

\begin{example}
\label{Example1}Let $0<r<p<\infty $, a weight $\omega ^{p}\in A_{\frac{p}{r}%
}(\mathbb{R}^{n})$ and $\{s_{k}\}=\{2^{ks}\omega ^{p}\}_{k\in \mathbb{Z}}$, $%
s\in \mathbb{R}$. Clearly, $\{s_{k}\}_{k\in \mathbb{Z}}$ lies in $\dot{X}%
_{\alpha ,\sigma ,p}$ for $\alpha _{1}=\alpha _{2}=s$, $\sigma =(r(\frac{p}{r%
})^{\prime },p)$.
\end{example}

\begin{remark}
\label{Tyulenev-class-properties}Let $0<\theta \leq p\leq \infty $. Let $%
\alpha _{1}$, $\alpha _{2}\in \mathbb{R}$, $\sigma _{1},\sigma _{2}\in
(0,+\infty ]$, $\sigma _{2}\geq p$, $\alpha =(\alpha _{1},\alpha _{2})$ and
let $\sigma =(\sigma _{1}=\theta \left( \frac{p}{\theta }\right) ^{\prime
},\sigma _{2})$. Let a $p$-admissible weight sequence $\{t_{k}\}\in \dot{X}%
_{\alpha ,\sigma ,p}$. Then $\alpha _{2}\geq \alpha _{1}$, see \cite{D20}.
\end{remark}

As usual, we put%
\begin{equation*}
\mathcal{M(}f)(x):=\sup_{Q}\frac{1}{|Q|}\int_{Q}\left\vert f(y)\right\vert
dy,\quad f\in L_{1}^{\mathrm{loc}}(\mathbb{R}^{n}),
\end{equation*}%
where the supremum\ is taken over all cubes with sides parallel to the axis
and $x\in Q$. Also we set 
\begin{equation*}
\mathcal{M}_{\sigma }(f):=\left( \mathcal{M(}\left\vert f\right\vert
^{\sigma })\right) ^{\frac{1}{\sigma }},\quad 0<\sigma <\infty .
\end{equation*}

In several situations we will use the following important tool, see \cite%
{D20}.

\begin{lemma}
\label{key-estimate1}Let $1<\theta \leq p<\infty $. Let $\{t_{k}\}$\ be a $p$%
-admissible\ weight\ sequence\ such that $t_{k}^{p}\in A_{\frac{p}{\theta }}(%
\mathbb{R}^{n})$, $k\in \mathbb{Z}$. Assume that $t_{k}^{p}$,\ $k\in \mathbb{%
Z}$ have the same Muckenhoupt constant, $A_{\frac{p}{\theta }%
}(t_{k}^{p})=c,k\in \mathbb{Z}$. Then%
\begin{equation}
\big\|\mathcal{M(}f_{k})|L_{p}(\mathbb{R}^{n},t_{k})\big\|\leq c\big\|%
f_{k}|L_{p}(\mathbb{R}^{n},t_{k})\big\|  \label{key-est.b1}
\end{equation}%
holds for all sequences $f_{k}\in L_{p}(\mathbb{R}^{n},t_{k})$, $k\in 
\mathbb{Z}$, where $c>0$ is independent of $k$.
\end{lemma}

\begin{remark}
\label{r-estimates}$\mathrm{(i)}$ We would like to mention that the result
of this lemma is true if we assume that $t_{k}^{p}\in A_{\frac{p}{\theta }}(%
\mathbb{R}^{n})$,\ $k\in \mathbb{Z}$, $1<p<\infty $ with $A_{\frac{p}{\theta 
}}(t_{k}^{p})\leq c$, $k\in \mathbb{Z}$, where $c>0$ independent of $k$. $%
\newline
\mathrm{(ii)}$ The property $\mathrm{\eqref{key-est.b1}}$ can be generalized
in the following way. Let $1<\theta \leq p<\infty $ and $\{t_{k}\}$ be a $p$%
-admissible sequence such that $t_{k}^{p}\in A_{\frac{p}{\theta }}(\mathbb{R}%
^{n})$, $k\in \mathbb{Z}$. \newline
$\mathrm{\bullet }$ If $t_{k}^{p}$,\ $k\in \mathbb{Z}$ satisfying $\mathrm{%
\eqref{Asum1}}$, then%
\begin{equation*}
\big\|\mathcal{M(}f_{j})|L_{p}(\mathbb{R}^{n},t_{k})\big\|\leq c\text{ }%
2^{\alpha _{1}(k-j)}\big\|f_{j}|L_{p}(\mathbb{R}^{n},t_{j})\big\|
\end{equation*}%
holds for all sequence of functions $f_{j}\in L_{p}(\mathbb{R}^{n},t_{j})$, $%
j\in \mathbb{Z}$ and $j\geq k$, where $c>0$ is independent of $k$ and $j$.%
\newline
$\mathrm{\bullet }$ If $t_{k}^{p}$,\ $k\in \mathbb{Z}$ satisfying $\mathrm{%
\eqref{Asum2}}$ with $\sigma _{2}\geq p$, then%
\begin{equation*}
\big\|\mathcal{M(}f_{j})|L_{p}(\mathbb{R}^{n},t_{k})\big\|\leq c\text{ }%
2^{\alpha _{2}(k-j)}\big\|f_{j}|L_{p}(\mathbb{R}^{n},t_{j})\big\|
\end{equation*}%
holds for all sequence of functions $f_{j}\in L_{p}(\mathbb{R}^{n},t_{j})$, $%
j\in \mathbb{Z}$ and $k\geq j$, where $c>0$ is independent of $k$ and $j$.$%
\newline
\mathrm{(iii)}$ A proof of this result\ for $t_{k}^{p}=\omega $, $k\in 
\mathbb{Z}$ may be found in \cite{Mu72}.$\newline
\mathrm{(iv)}$ In view of Lemma \ref{Ap-Property}/(iv) we can assume that $%
t_{k}^{p}\in A_{p}(\mathbb{R}^{n})$,\ $k\in \mathbb{Z}$, $1<p<\infty $ with $%
A_{p}(t_{k}^{p})\leq c$, $k\in \mathbb{Z}$, where $c>0$ independent of $k$.
\end{remark}

We state one of the main tools of this paper, see \cite{D20.1}.

\begin{lemma}
\label{FS-inequality}Let $1<\theta \leq p<\infty $\ and $1<q<\infty $. Let $%
\{t_{k}\}$\ be a $p$-admissible\ weight\ sequence\ such that $t_{k}^{p}\in
A_{\frac{p}{\theta }}(\mathbb{R}^{n})$, $k\in \mathbb{Z}$. Assume that $%
t_{k}^{p}$,\ $k\in \mathbb{Z}$ have the same Muckenhoupt constant, $A_{\frac{%
p}{\theta }}(t_{k}^{p})=c,k\in \mathbb{Z}$. Then%
\begin{equation*}
\Big\|\Big(\sum\limits_{k=-\infty }^{\infty }t_{k}^{q}\big(\mathcal{M(}f_{k})%
\big)^{q}\Big)^{\frac{1}{q}}|L_{p}(\mathbb{R}^{n})\Big\|\lesssim \Big\|\Big(%
\sum\limits_{k=-\infty }^{\infty }t_{k}^{q}\left\vert f_{k}\right\vert ^{q}%
\Big)^{\frac{1}{q}}|L_{p}(\mathbb{R}^{n})\Big\|
\end{equation*}%
holds for all sequences of functions $\{f_{k}\}\in L_{p}(\ell _{q})$.
\end{lemma}

\begin{remark}
\label{r-estimates copy(1)}$\mathrm{(i)}$ We would like to mention that the
result of this lemma is true if we assume that $t_{k}^{p}\in A_{\frac{p}{%
\theta }}(\mathbb{R}^{n})$,\ $k\in \mathbb{Z}$, $1<p<\infty $ with $A_{\frac{%
p}{\theta }}(t_{k}^{p})\leq c,k\in \mathbb{Z}$, where $c>0$ independent of $%
k $.$\newline
\mathrm{(ii)}$ In view of Lemma \ref{Ap-Property}/(iv) we can assume that $%
t_{k}^{p}\in A_{p}(\mathbb{R}^{n})$,\ $k\in \mathbb{Z}$, $1<p<\infty $ with $%
A_{p}(t_{k}^{p})\leq c,k\in \mathbb{Z}$, where $c>0$ independent of $k$.
\end{remark}

\section{Function spaces}

In this section we\ present the Fourier analytical definition of
Triebel-Lizorkin spaces of variable smoothness and recall some their
properties. Our goal here is to study the spaces $\dot{F}_{\infty ,q}(%
\mathbb{R}^{n},\{t_{k}\})$ where their basic properties are given in analogy
to the Triebel-Lizorkin spaces $\dot{F}_{\infty ,q}(\mathbb{R}^{n})$. Select
a pair of Schwartz functions $\varphi $ and $\psi $ satisfy%
\begin{equation}
\text{supp}\mathcal{F}\varphi ,\mathcal{F}\psi \subset \big\{\xi :\frac{1}{2}%
\leq |\xi |\leq 2\big\},  \label{Ass1}
\end{equation}%
\begin{equation}
|\mathcal{F}\varphi (\xi )|,|\mathcal{F}\psi (\xi )|\geq c\quad \text{if}%
\quad \frac{3}{5}\leq |\xi |\leq \frac{5}{3}  \label{Ass2}
\end{equation}%
and 
\begin{equation}
\sum_{k=-\infty }^{\infty }\overline{\mathcal{F}\varphi (2^{-k}\xi )}%
\mathcal{F}\psi (2^{-k}\xi )=1\quad \text{if}\quad \xi \neq 0,  \label{Ass3}
\end{equation}%
where $c>0$. Throughout the paper, for all $k$ $\in \mathbb{Z}$ and $x\in 
\mathbb{R}^{n}$, we put $\varphi _{k}(x):=2^{kn}\varphi (2^{k}x)$ and $%
\tilde{\varphi}(x):=\overline{\varphi (-x)}$. Let $\varphi \in \mathcal{S}(%
\mathbb{R}^{n})$ be a function satisfying $\mathrm{\eqref{Ass1}}$-$\mathrm{%
\eqref{Ass2}}$. We recall that there exists a function $\psi \in \mathcal{S}(%
\mathbb{R}^{n})$ satisfying $\mathrm{\eqref{Ass1}}$-$\mathrm{\eqref{Ass3}}$,
see \cite[Lemma (6.9)]{FrJaWe01}.

We start by recalling the definition of $\dot{F}_{p,q}(\mathbb{R}%
^{n},\{t_{k}\})$ spaces.

\begin{definition}
\label{B-F-def}Let $0<p<\infty $ and $0<q\leq \infty $. Let $\{t_{k}\}$ be a 
$p$-admissible weight sequence, and $\varphi \in \mathcal{S}(\mathbb{R}^{n})$%
\ satisfy $\mathrm{\eqref{Ass1}}$ and $\mathrm{\eqref{Ass2}}$. The
Triebel-Lizorkin space $\dot{F}_{p,q}(\mathbb{R}^{n},\{t_{k}\})$\ is the
collection of all $f\in \mathcal{S}_{\infty }^{\prime }(\mathbb{R}^{n})$\
such that 
\begin{equation*}
\big\|f|\dot{F}_{p,q}(\mathbb{R}^{n},\{t_{k}\})\big\|:=\Big\|\Big(%
\sum\limits_{k=-\infty }^{\infty }t_{k}^{q}|\varphi _{k}\ast f|^{q}\Big)^{%
\frac{1}{q}}|L_{p}(\mathbb{R}^{n})\Big\|<\infty
\end{equation*}%
with the usual modifications if $q=\infty $.
\end{definition}

\begin{remark}
Some properties of these function spaces, such as the $\varphi $-transform
characterization in the sense of Frazier and Jawerth, the smooth atomic\ and
molecular decomposition and the characterization of these function spaces in
terms of the difference relations are given in\ \cite{D20} and \cite{D20.1}.
\end{remark}

As in \cite[Section 5]{FJ90}, we introduce the following function spaces.

\begin{definition}
\label{F-inf-def}Let $0<q<\infty $.\ Let $\{t_{k}\}$ be a $q$-admissible
weight sequence and $\varphi \in \mathcal{S}(\mathbb{R}^{n})$\ satisfy $%
\mathrm{\eqref{Ass1}}$ and $\mathrm{\eqref{Ass2}}$. The Triebel-Lizorkin\
space $\dot{F}_{\infty ,q}(\mathbb{R}^{n},\{t_{k}\})$\ is the collection of
all $f\in \mathcal{S}_{\infty }^{\prime }(\mathbb{R}^{n})$\ such that 
\begin{equation*}
{\big\|}f|\dot{F}_{\infty ,q}(\mathbb{R}^{n},\{t_{k}\}){\big\|}:=\sup_{P\in 
\mathcal{Q}}\Big(\frac{1}{|P|}\int_{P}\sum\limits_{k=-\log _{2}l(P)}^{\infty
}t_{k}^{q}(x)|\varphi _{k}\ast f(x)|^{q}dx\Big)^{\frac{1}{q}}<\infty .
\end{equation*}
\end{definition}

\begin{remark}
We would like to mention that the elements of the above spaces are not
distributions but equivalence classes of distributions$.$
\end{remark}

Using the system $\{\varphi _{k}\}_{k\in \mathbb{Z}}$ we can define the
quasi-norms%
\begin{equation*}
\big\|f|\dot{F}_{p,q}^{s}(\mathbb{R}^{n})\big\|:=\big\|\Big(%
\sum\limits_{k=-\infty }^{\infty }2^{ksq}|\varphi _{k}\ast f|^{q}\Big)^{%
\frac{1}{q}}|L_{p}(\mathbb{R}^{n})\big\|
\end{equation*}%
for constants $s\in \mathbb{R}$ and $0<p,q\leq \infty $. The
Triebel-Lizorkin space $\dot{F}_{p,q}^{s}(\mathbb{R}^{n})$\ consist of all
distributions $f\in \mathcal{S}_{\infty }^{\prime }(\mathbb{R}^{n})$ for
which $\big\|f|\dot{F}_{p,q}^{s}(\mathbb{R}^{n})\big\|<\infty $. Further
details on the classical theory of these spaces can be found in \cite{FJ86}, 
\cite{FJ90}, \cite{FrJaWe01}, \cite{T1} and \cite{T2}.

One recognizes immediately that if $\{t_{k}\}=\{2^{sk}\}$, $s\in \mathbb{R}$%
, then 
\begin{equation*}
\dot{F}_{p,q}(\mathbb{R}^{n},\{2^{sk}\})=\dot{F}_{p,q}^{s}(\mathbb{R}%
^{n})\quad \text{and}\quad \dot{F}_{\infty ,q}(\mathbb{R}^{n},\{2^{sk}\})=%
\dot{F}_{\infty ,q}^{s}(\mathbb{R}^{n}).
\end{equation*}%
Moreover, for $\{t_{k}\}=\{2^{sk}w\}$, $s\in \mathbb{R}$ with a weight $w$
we re-obtain the weighted Triebel-Lizorkin spaces; we refer to the papers 
\cite{Bui82},\ \cite{IzSa12}\ and \cite{Tang} for a comprehensive treatment
of the weighted spaces.

A basic tool to study\ the above\ function\ spaces is the following Calder%
\'{o}n reproducing formula, see \cite[Lemma 2.1]{YY2}.

\begin{lemma}
\label{DW-lemma1}Suppose $\varphi $, $\psi \in \mathcal{S}(\mathbb{R}^{n})$\
satisfying $\mathrm{\eqref{Ass1}}$ through $\mathrm{\eqref{Ass3}}$\textrm{. }%
If\textrm{\ }$f\in \mathcal{S}_{\infty }^{\prime }(\mathbb{R}^{n})$, then%
\begin{equation}
f=\sum_{k=-\infty }^{\infty }2^{-kn}\sum_{m\in \mathbb{Z}^{n}}\widetilde{%
\varphi }_{k}\ast f(2^{-k}m)\psi _{k}(\cdot -2^{-k}m).  \label{proc2}
\end{equation}
\end{lemma}

Let $\varphi $, $\psi \in \mathcal{S}(\mathbb{R}^{n})$ satisfying $\mathrm{%
\eqref{Ass1}}$\ through\ $\mathrm{\eqref{Ass3}}$. Recall that the $\varphi $%
-transform $S_{\varphi }$ is defined by setting $(S_{\varphi
}f)_{k,m}=\langle f,\varphi _{k,m}\rangle $ where $\varphi
_{k,m}(x)=2^{kn/2}\varphi (2^{k}x-m)$, $m\in \mathbb{Z}^{n}$ and $k\in 
\mathbb{Z}$. The inverse $\varphi $-transform $T_{\psi }$ is defined by 
\begin{equation}
T_{\psi }\lambda :=\sum_{k=-\infty }^{\infty }\sum_{m\in \mathbb{Z}%
^{n}}\lambda _{k,m}\psi _{k,m},  \label{inv-phi-tran-serie}
\end{equation}%
where $\lambda =\{\lambda _{k,m}\}_{k\in \mathbb{Z},m\in \mathbb{Z}%
^{n}}\subset \mathbb{C}$, see \cite{FJ90}.

Now we introduce the following sequence spaces.

\begin{definition}
\label{sequence-space}Let $0<p<\infty $ and $0<q\leq \infty $. Let $%
\{t_{k}\} $ be a $p$-admissible weight sequence. Then for all complex valued
sequences $\lambda =\{\lambda _{k,m}\}_{k\in \mathbb{Z},m\in \mathbb{Z}%
^{n}}\subset \mathbb{C}$ we define 
\begin{equation*}
\dot{f}_{p,q}(\mathbb{R}^{n},\{t_{k}\}):=\Big\{\lambda :\big\|\lambda |\dot{f%
}_{p,q}(\mathbb{R}^{n},\{t_{k}\})\big\|<\infty \Big\}
\end{equation*}%
where%
\begin{equation*}
\big\|\lambda |\dot{f}_{p,q}(\mathbb{R}^{n},\{t_{k}\})\big\|:=\Big\|\Big(%
\sum_{k=-\infty }^{\infty }\sum\limits_{m\in \mathbb{Z}^{n}}2^{\frac{knq}{2}%
}t_{k}^{q}|\lambda _{k,m}|^{q}\chi _{k,m}\Big)^{\frac{1}{q}}|L_{p}(\mathbb{R}%
^{n})\Big\|.
\end{equation*}
\end{definition}

Allowing the smoothness $t_{k}$, $k\in \mathbb{Z}$ to vary from point to
point will raise extra difficulties\ to study these function spaces. But by
the following lemma the problem can be reduced to the case of fixed
smoothness, see \cite{D23}.

\begin{proposition}
\label{Equi-norm1} Let $0<\theta \leq p<\infty $, $0<q<\infty $, $0<\delta
\leq 1$ and $\{t_{k}\}$ be a $p$-admissible sequence. Assume that $\{t_{k}\}$
satisfying $\mathrm{\eqref{Asum1}}$ with $\sigma _{1}=\theta \left( \frac{p}{%
\theta }\right) ^{\prime }$ and $j=k$. Then%
\begin{equation*}
\big\|\lambda |\dot{f}_{p,q,\delta }(\mathbb{R}^{n},\{t_{k}\})\big\|^{\ast
}:=\Big\|\Big(\sum_{k=-\infty }^{\infty }\sum\limits_{m\in \mathbb{Z}%
^{n}}2^{knq(\frac{1}{2}+\frac{1}{\delta p})}t_{k,m,\delta p}^{q}|\lambda
_{k,m}|^{q}\chi _{k,m}\Big)^{\frac{1}{q}}|L_{p}(\mathbb{R}^{n})\Big\|,
\end{equation*}%
is an equivalent quasi-norm in $\dot{f}_{p,q}(\mathbb{R}^{n},\{t_{k}\})$,
where%
\begin{equation*}
t_{k,m,\delta p}:=\big\|t_{k}|L_{\delta p}(Q_{k,m})\big\|,\quad k\in \mathbb{%
Z},m\in \mathbb{Z}^{n}.
\end{equation*}
\end{proposition}

We define $\dot{f}_{\infty ,q}(\mathbb{R}^{n},\{t_{k}\})$, the sequence
space corresponding to\ $\dot{F}_{\infty ,q}(\mathbb{R}^{n},\{t_{k}\})$ as
follows.

\begin{definition}
Let $0<q<\infty $ and $\{t_{k}\}$ be a $q$-admissible sequence. Then for all
complex valued sequences $\lambda =\{\lambda _{k,m}\}_{k\in \mathbb{Z},m\in 
\mathbb{Z}^{n}}\subset \mathbb{C}$ we define 
\begin{equation*}
\dot{f}_{\infty ,q}(\mathbb{R}^{n},\{t_{k}\}):=\Big\{\lambda :{\big\|}%
\lambda |\dot{f}_{\infty ,q}(\mathbb{R}^{n},\{t_{k}\}){\big\|}<\infty \Big\},
\end{equation*}%
where%
\begin{equation}
{\big\|}\lambda |\dot{f}_{\infty ,q}(\mathbb{R}^{n},\{t_{k}\}){\big\|}%
:=\sup_{P\in \mathcal{Q}}\Big(\frac{1}{|P|}\int_{P}\sum\limits_{k=-\log
_{2}l(P)}^{\infty }\sum\limits_{m\in \mathbb{Z}^{n}}2^{\frac{knq}{2}%
}t_{k}^{q}(x)|\lambda _{k,m}|^{q}\chi _{k,m}(x)dx\Big)^{\frac{1}{q}}.
\label{norm}
\end{equation}
\end{definition}

The quasi-norm $\mathrm{\eqref{norm}}$ can be rewritten as follows:

\begin{proposition}
Let $0<q<\infty $. Let $\{t_{k}\}$ be a $q$-admissible sequence. Then 
\begin{equation}
\big\|\lambda |\dot{f}_{\infty ,q}(\mathbb{R}^{n},\{t_{k}\})\big\|%
=\sup_{P\in \mathcal{Q}}\Big(\frac{1}{|P|}\int_{P}\sum\limits_{k=-\log
_{2}l(P)}^{\infty }\sum\limits_{m\in \mathbb{Z}^{n}}2^{knq(\frac{1}{2}+\frac{%
1}{q})}t_{k,m,q}^{q}|\lambda _{k,m}|^{q}\chi _{k,m}(x)dx\Big)^{\frac{1}{q}}.
\label{equi -f-inf-sequence}
\end{equation}
\end{proposition}

\begin{lemma}
\label{convergence}Let $\alpha =(\alpha _{1},\alpha _{2})\in \mathbb{R}%
^{2},0<\theta \leq q<\infty $ and $\{t_{k}\}\in \dot{X}_{\alpha ,\sigma ,q}$
be a $q$-admissible weight sequence with $\sigma =(\sigma _{1}=\theta \left( 
\frac{q}{\theta }\right) ^{\prime },\sigma _{2}\geq q)$. Let\ $\psi \in 
\mathcal{S}(\mathbb{R}^{n})$ satisfying $\mathrm{\eqref{Ass1}}$ and $\mathrm{%
\eqref{Ass2}}$. Then for all $\lambda \in \dot{f}_{\infty ,q}(\mathbb{R}%
^{n},\{t_{k}\})$%
\begin{equation*}
T_{\psi }\lambda :=\sum_{k=-\infty }^{\infty }\sum_{m\in \mathbb{Z}%
^{n}}\lambda _{k,m}\psi _{v,m},
\end{equation*}%
converges in $\mathcal{S}_{\infty }^{\prime }(\mathbb{R}^{n})$; moreover, $%
T_{\psi }:\dot{f}_{\infty ,q}(\mathbb{R}^{n},\{t_{k}\})\rightarrow \mathcal{S%
}_{\infty }^{\prime }(\mathbb{R}^{n})$ is continuous.
\end{lemma}

\begin{proof}
Let $\lambda \in \dot{f}_{\infty ,q}(\mathbb{R}^{n},\{t_{k}\})$ and $\varphi
\in \mathcal{S}_{\infty }(\mathbb{R}^{n})$. We see that%
\begin{equation*}
\sum_{k=-\infty }^{\infty }\sum_{m\in \mathbb{Z}^{n}}|\lambda
_{k,m}||\langle \psi _{k,m},\varphi \rangle |=I_{1}+I_{2},
\end{equation*}%
where%
\begin{equation*}
I_{1}=\sum_{k=-\infty }^{0}\sum_{m\in \mathbb{Z}^{n}}|\lambda
_{k,m}||\langle \psi _{k,m},\varphi \rangle |,\quad I_{2}=\sum_{k=1}^{\infty
}\sum_{m\in \mathbb{Z}^{n}}|\lambda _{k,m}||\langle \psi _{k,m},\varphi
\rangle |.
\end{equation*}%
It suffices to show that both $I_{1}$ and $I_{2}$ are dominated by%
\begin{equation*}
c\big\|\lambda |\dot{f}_{\infty ,q}(\mathbb{R}^{n},\{t_{k}\})\big\|.
\end{equation*}%
\textit{Estimation of }$I_{1}$. Let us recall the following estimate, see
(3.18) in \cite{M07}. For any $L>0$, there exists a positive constant $M\in 
\mathbb{N}$ such that for all $\varphi $, $\psi \in \mathcal{S}_{\infty }(%
\mathbb{R}^{n}),i,k\in \mathbb{Z}$ and $m,h\in \mathbb{Z}^{n}$,%
\begin{equation*}
\left\vert \langle \varphi _{k,m},\psi _{i,h}\rangle \right\vert \lesssim {%
\big\|}\varphi {\big\|}_{\mathcal{S}_{M}}{\big\|}\psi {\big\|}_{\mathcal{S}%
_{M}}\Big(1+\frac{|2^{-k}m-2^{-i}h|^{n}}{\max (2^{-kn},2^{-in})}\Big)%
^{-L}\min \left( 2^{(i-k)nL},2^{(k-i)nL}\right) .
\end{equation*}%
Therefore,%
\begin{equation*}
\left\vert \langle \psi _{k,m},\varphi \rangle \right\vert \lesssim {\big\|}%
\varphi {\big\|}_{\mathcal{S}_{M}}{\big\|}\psi {\big\|}_{\mathcal{S}_{M}}%
\Big(1+\frac{|2^{-k}m|^{n}}{\max (1,2^{-kn})}\Big)^{-L}2^{-|k|nL},
\end{equation*}%
where the implicit constant is independent of $i,k\in \mathbb{Z},m,h\in 
\mathbb{Z}^{n}$ and 
\begin{equation*}
{\big\|}\varphi {\big\|}_{\mathcal{S}_{M}}=\sup_{x\in \mathbb{R}%
^{n}}\sup_{|\alpha |\leq M}(1+|x|)^{M}|\partial ^{\alpha }\varphi (x)|.
\end{equation*}%
Our estimate use partially some decomposition techniques already used in 
\cite{FJ90} and \cite{Ky03}. {For each }$j\in \mathbb{N}${\ we define }%
\begin{equation*}
{\Omega _{j}:=\{m\in \mathbb{Z}^{n}:2^{j-1}<|m|\leq 2^{j}\}}\text{\quad
and\quad }{\Omega _{0}:=\{m\in \mathbb{Z}^{n}:|m|\leq 1\}.}
\end{equation*}%
Thus,%
\begin{align}
I_{1}\lesssim & {\big\|}\varphi {\big\|}_{\mathcal{S}_{M}}\sum_{k=-\infty
}^{0}2^{knL}\sum_{m\in \mathbb{Z}^{n}}\frac{|\lambda _{k,m}|}{\big(1+|m|\big)%
^{nL}}  \notag \\
=& c{\big\|}\varphi {\big\|}_{\mathcal{S}_{M}}\sum_{k=-\infty
}^{0}2^{knL}\sum\limits_{j=0}^{\infty }\sum\limits_{m\in \Omega _{j}}\frac{%
|\lambda _{k,m}|}{\big(1+|m|\big)^{nL}}  \notag \\
\lesssim & {\big\|}\varphi {\big\|}_{\mathcal{S}_{M}}\sum_{k=-\infty
}^{0}2^{knL}\sum\limits_{j=0}^{\infty }2^{-nLj}\sum\limits_{m\in \Omega
_{j}}|\lambda _{k,m}|.  \notag
\end{align}%
Let $0<\varrho <1$ be such that $\frac{1}{\varrho }=\frac{1}{\tau }+\frac{1}{%
\sigma _{1}}$ with $0<\tau <\min \big(q,\frac{1}{\max (0,1-\frac{1}{\sigma
_{1}})}\big)$. Using the embedding $\ell _{\varrho }\hookrightarrow \ell
_{1} $ we find that 
\begin{align*}
I_{1}\lesssim & {\big\|}\varphi {\big\|}_{\mathcal{S}_{M}}\sum_{k=-\infty
}^{0}2^{knL}\sum\limits_{j=0}^{\infty }2^{-nLj}\Big(\sum\limits_{m\in \Omega
_{j}}|\lambda _{k,m}|^{\varrho }\Big)^{\frac{1}{\varrho }} \\
=& c{\big\|}\varphi {\big\|}_{\mathcal{S}_{M}}\sum_{k=-\infty
}^{0}2^{knL}\sum\limits_{j=0}^{\infty }2^{(\frac{1}{\varrho }-L)nj}\Big(%
2^{(k-j)n}\int\limits_{\cup _{z\in \Omega _{j}}Q_{k,z}}\sum\limits_{m\in
\Omega _{j}}|\lambda _{k,m}|^{\varrho }\chi _{k,m}(y)dy\Big)^{\frac{1}{%
\varrho }}.
\end{align*}%
Let $y\in \cup _{z\in \Omega _{j}}Q_{k,z}$ and $x\in Q_{0,0}$. Then $y\in
Q_{k,z}$ for some $z\in \Omega _{j}$ and ${2^{j-1}<|z|\leq 2^{j}}$. From
this it follows that%
\begin{align*}
\left\vert y-x\right\vert \leq & \left\vert y-2^{-k}z\right\vert +\left\vert
x-2^{-k}z\right\vert \\
\leq & \sqrt{n}\text{ }2^{-k}+\left\vert x\right\vert +2^{-k}\left\vert
z\right\vert \\
\leq & 2^{j-k+\delta _{n}},\quad \delta _{n}\in \mathbb{N},
\end{align*}%
which implies that $y$ is located in the ball $B(x,2^{j-k+\delta _{n}})$. In
addition, from the fact that%
\begin{equation*}
\left\vert y\right\vert \leq \left\vert y-x\right\vert +\left\vert
x\right\vert \leq 2^{j-k+\delta _{n}}+1\leq 2^{j-k+c_{n}},\quad c_{n}\in 
\mathbb{N},
\end{equation*}%
we have that $y$ is located in the ball $B(0,2^{j-k+c_{n}})$. Therefore, 
\begin{align*}
& \Big(2^{(k-j)n}\int\limits_{\cup _{z\in \Omega
_{j}}Q_{k,z}}\sum\limits_{m\in \Omega _{j}}|\lambda _{k,m}|^{\varrho }\chi
_{k,m}(y)dy\Big)^{\frac{1}{\varrho }} \\
\leq & \Big(2^{(k-j)n}\int\limits_{B(x,2^{j-k+c_{n}})}\sum\limits_{m\in
\Omega _{j}}|\lambda _{k,m}|^{\tau }t_{k}^{\tau }(y)\chi _{k,m}(y)\chi
_{B(0,2^{j-k+c_{n}})}(y)dy\Big)^{\frac{1}{\tau }}M_{B(0,2^{j-k+c_{n}}),%
\sigma _{1}}(t_{k}^{-1}) \\
\lesssim & \mathcal{M}_{\tau }\big(\sum\limits_{m\in \mathbb{Z}%
^{n}}t_{k}\left\vert \lambda _{k,m}\right\vert \chi _{k,m}\chi
_{B(0,2^{j-k+c_{n}})}\big)(x)M_{B(0,2^{j-k+c_{n}}),\sigma _{1}}(t_{k}^{-1}).
\end{align*}%
By Lemma {\ref{Ap-Property}/(i)-(ii)}, $\mathrm{\eqref{Asum1}}$ and $\mathrm{%
\eqref{Asum2}}$ we obtain that $t_{k}^{-\sigma _{1}}\in A_{(\frac{q}{\theta }%
)^{\prime }}(\mathbb{R}^{n})$, $k\in \mathbb{Z}$ and%
\begin{align*}
M_{B(0,2^{j-k+c_{n}}),\sigma _{1}}(t_{k}^{-1})\lesssim & 2^{(j-k)\frac{n}{q}%
}M_{B(0,1),\sigma _{1}}(t_{k}^{-1}) \\
\lesssim & 2^{(j-k)\frac{n}{q}}\left( M_{B(0,1),p}(t_{k})\right) ^{-1} \\
\lesssim & 2^{(j-k)\frac{n}{q}-k\alpha _{2}}\left(
M_{B(0,1),p}(t_{0})\right) ^{-1}
\end{align*}%
for any $k\leq 0$ and any $j\in \mathbb{N}_{0}$. Therefore, for any $L$
large enough, 
\begin{equation*}
I_{1}\lesssim {\big\|}\varphi {\big\|}_{\mathcal{S}_{M}}\sum_{k=-\infty
}^{0}2^{k(nL-\alpha _{2}-\frac{n}{q})}\sum\limits_{j=0}^{\infty }2^{(\frac{1%
}{\varrho }-L+\frac{n}{q})nj}\mathcal{M}_{\tau }\big(\sum\limits_{m\in 
\mathbb{Z}^{n}}t_{k}\lambda _{k,m}\chi _{k,m}\chi _{B(0,2^{j-k+c_{n}})}\big)%
(x)
\end{equation*}%
for any $x\in Q_{0,0}$. Using Lemma {\ref{key-estimate1}}, we obtain%
\begin{align*}
& \big\|\mathcal{M}_{\tau }\big(\sum\limits_{m\in \mathbb{Z}%
^{n}}t_{k}\lambda _{k,m}\chi _{k,m}\chi _{B(0,2^{j-k+c_{n}})}\big)%
|L_{q}(Q_{0,0})\big\| \\
\lesssim & \big\|\sum\limits_{m\in \mathbb{Z}^{n}}t_{k}\lambda _{k,m}\chi
_{k,m}\chi _{B(0,2^{j-k+c_{n}})}|L_{q}(\mathbb{R}^{n})\big\| \\
\lesssim & 2^{(j-k)\frac{n}{q}}\big\|\lambda |\dot{f}_{\infty ,q}(\mathbb{R}%
^{n},\{t_{k}\})\big\|
\end{align*}%
for any $k\leq 0$. Indeed, we have%
\begin{align*}
& \big\|\sum\limits_{m\in \mathbb{Z}^{n}}t_{k}\lambda _{k,m}\chi _{k,m}\chi
_{B(0,2^{j-k+c_{n}})}|L_{q}(\mathbb{R}^{n})\big\| \\
=& c2^{(j-k)\frac{n}{q}}\Big(\frac{1}{|B(0,2^{j-k+c_{n}})|}%
\int_{B(0,2^{j-k+c_{n}})}\sum\limits_{m\in \mathbb{Z}^{n}}t_{k}^{q}(x)|%
\lambda _{k,m}|^{q}\chi _{k,m}(x)dx\Big)^{\frac{1}{q}} \\
\lesssim & 2^{(j-k)\frac{n}{q}}\Big(\frac{1}{|B(0,2^{j-k+c_{n}})|}%
\int_{B(0,2^{j-k+c_{n}})}\sum_{i=k-j}^{\infty }\sum\limits_{m\in \mathbb{Z}%
^{n}}t_{i}^{q}(x)|\lambda _{i,m}|^{q}\chi _{i,m}(x)dx\Big)^{\frac{1}{q}} \\
\lesssim & 2^{(j-k)\frac{n}{q}}\big\|\lambda |\dot{f}_{\infty ,q}(\mathbb{R}%
^{n},\{t_{k}\})\big\|.
\end{align*}%
Taking $L$ large enough we obtain%
\begin{equation*}
I_{1}\lesssim {\big\|}\varphi {\big\|}_{\mathcal{S}_{M}}\big\|\lambda |\dot{f%
}_{\infty ,q}(\mathbb{R}^{n},\{t_{k}\})\big\|.
\end{equation*}%
\textit{Estimation of} $I_{2}$. {For each }$j\in \mathbb{N}${\ we define }%
\begin{equation*}
{\Omega _{j}:=\{h\in \mathbb{Z}^{n}:2^{j+k-1}<|h|\leq 2^{j+k}\}\quad }\text{{%
and}}{\quad \Omega _{0}:=\{h\in \mathbb{Z}^{n}:|h|\leq 2^{k}\}.}
\end{equation*}%
{\ }Then we find%
\begin{align}
I_{2}\lesssim & {\big\|}\varphi {\big\|}_{\mathcal{S}_{M}}\sum_{k=1}^{\infty
}2^{-knL}\sum_{m\in \mathbb{Z}^{n}}\frac{|\lambda _{k,m}|}{\big(1+|2^{-k}m|%
\big)^{nL}}  \notag \\
=& c{\big\|}\varphi {\big\|}_{\mathcal{S}_{M}}\sum_{k=1}^{\infty
}2^{-knL}\sum\limits_{j=0}^{\infty }\sum\limits_{m\in \Omega _{j}}\frac{%
|\lambda _{k,m}|}{\big(1+|2^{-k}m|\big)^{nL}}  \notag \\
\leq & c{\big\|}\varphi {\big\|}_{\mathcal{S}_{M}}\sum_{k=1}^{\infty
}2^{-knL}\sum\limits_{j=0}^{\infty }2^{-nLj}\sum\limits_{m\in \Omega
_{j}}|\lambda _{k,m}|.  \notag
\end{align}%
Let $\varrho $ and $\tau $ be as in the estimation of $I_{1}$. The embedding 
$\ell _{\varrho }\hookrightarrow \ell _{1}$ yields that 
\begin{align*}
I_{2}\lesssim & {\big\|}\varphi {\big\|}_{\mathcal{S}_{M}}\sum_{k=1}^{\infty
}2^{-knL}\sum\limits_{j=0}^{\infty }2^{-Lj}\big(\sum\limits_{m\in \Omega
_{j}}|\lambda _{k,m}|^{\varrho }\big)^{\frac{1}{\varrho }} \\
=& c{\big\|}\varphi {\big\|}_{\mathcal{S}_{M}}\sum_{k=1}^{\infty
}2^{-knL}\sum\limits_{j=0}^{\infty }2^{(\frac{n}{\varrho }-nL)j}\big(%
2^{(k-j)n}\int\limits_{\cup _{z\in \Omega _{j}}Q_{k,z}}\sum\limits_{m\in
\Omega _{j}}|\lambda _{k,m}|^{\varrho }\chi _{k,m}(y)dy\big)^{\frac{1}{%
\varrho }}.
\end{align*}%
Let $y\in \cup _{z\in \Omega _{j}}Q_{k,z}$ and $x\in Q_{0,0}$. Then $y\in
Q_{k,z}$ for some $z\in \Omega _{j}$ and ${2^{j-1}<2^{-k}|z|\leq 2^{j}}$.
From this it follows that%
\begin{align*}
\left\vert y-x\right\vert \leq & \left\vert y-2^{-k}z\right\vert +\left\vert
x-2^{-k}z\right\vert \\
\leq & \sqrt{n}\text{ }2^{-k}+\left\vert x\right\vert +2^{-k}\left\vert
z\right\vert \\
\leq & 2^{j+\delta _{n}},\quad \delta _{n}\in \mathbb{N},
\end{align*}%
which implies that $y$ is located in the ball $B\left( x,2^{j+\delta
_{n}}\right) $. In addition, from the fact that%
\begin{equation*}
\left\vert y\right\vert \leq \left\vert y-x\right\vert +\left\vert
x\right\vert \leq 2^{j+\delta _{n}}+1\leq 2^{j+c_{n}},\quad c_{n}\in \mathbb{%
N},
\end{equation*}%
we have that $y$ is located in the ball $B(0,2^{j+c_{n}})$. Therefore, 
\begin{align*}
& \big(2^{(k-j)n}\int\limits_{\cup _{z\in \Omega
_{j}}Q_{k,z}}\sum\limits_{m\in \Omega _{j}}|\lambda _{k,m}|^{\varrho }\chi
_{k,m}(y)dy\big)^{\frac{1}{\varrho }} \\
\leq & 2^{k\frac{n}{\varrho }}\Big(2^{-jn}\int\limits_{B(x,2^{j+\delta
_{n}})}\sum\limits_{m\in \Omega _{j}}|\lambda _{k,m}|^{\tau }t_{k}^{\tau
}(y)\chi _{k,m}(y)\chi _{B(0,2^{j+c_{n}})}(y)dy\Big)^{\frac{1}{\tau }%
}M_{B(0,2^{j+c_{n}}),\sigma _{1}}(t_{k}^{-1}) \\
\lesssim & 2^{k\frac{n}{\varrho }}\mathcal{M}_{\tau }\big(\sum\limits_{m\in 
\mathbb{Z}^{n}}t_{k}\lambda _{k,m}\chi _{k,m}\chi _{B(0,2^{j+c_{n}})}\big)%
(x)M_{B(0,2^{j+c_{n}}),\sigma _{1}}(t_{k}^{-1}).
\end{align*}%
By $\mathrm{\eqref{Asum1}}$ and Lemma {\ref{Ap-Property}/(iii)}, 
\begin{align*}
M_{B(0,2^{j+c_{n}}),\sigma _{1}}(t_{k}^{-1})\lesssim & 2^{-k\alpha _{1}}\big(%
M_{B(0,2^{j+c_{n}}),p}(t_{0})\big)^{-1} \\
\lesssim & 2^{j(\frac{n}{p}-\frac{n\delta }{p})-k\alpha _{1}}\big(%
M_{B(0,1),p}(t_{0})\big)^{-1} \\
\lesssim & 2^{j(\frac{n}{p}-\frac{n\delta }{p})-k\alpha _{1}}\left(
M_{B(0,1),p}(t_{0})\right) ^{-1}.
\end{align*}%
Therefore, 
\begin{equation*}
I_{2}\lesssim {\big\|}\varphi {\big\|}_{\mathcal{S}_{M}}\sum_{k=1}^{\infty
}2^{-k(nL-\frac{n}{\varrho }+\alpha _{1})}\sum\limits_{j=0}^{\infty }2^{(%
\frac{n}{\varrho }-nL+\frac{n}{p}-\frac{n\delta }{p})j}\mathcal{M}_{\tau }%
\big(t_{k}\sum\limits_{m\in \mathbb{Z}^{n}}\lambda _{k,m}\chi _{k,m}\chi
_{B(0,2^{j+c_{n}})}\big)(x)
\end{equation*}%
for any $x\in Q_{0,0}$. As in the estimation of $I_{1}$, we obtain%
\begin{equation*}
I_{2}\lesssim {\big\|}\varphi {\big\|}_{\mathcal{S}_{M}}{\big\|}\lambda |%
\dot{f}_{\infty ,q}(\mathbb{R}^{n},\{t_{k}\}){\big\|}.
\end{equation*}%
This completes the proof of Lemma {\ref{convergence}.}
\end{proof}

For a sequence $\lambda =\{\lambda _{k,m}\}_{k\in \mathbb{Z},m\in \mathbb{Z}%
^{n}}\subset \mathbb{C},0<r\leq \infty $ and a fixed $d>0$, set%
\begin{equation*}
\lambda _{k,m,r,d}^{\ast }:=\Big(\sum_{h\in \mathbb{Z}^{n}}\frac{|\lambda
_{k,h}|^{r}}{(1+2^{k}|2^{-k}h-2^{-k}m|)^{d}}\Big)^{\frac{1}{r}}
\end{equation*}%
and $\lambda _{r,d}^{\ast }:=\{\lambda _{k,m,r,d}^{\ast }\}_{k\in \mathbb{Z}%
,m\in \mathbb{Z}^{n}}\subset \mathbb{C}$.

\begin{lemma}
\label{lamda-equi2}Let $0<\theta \leq q<\infty ,\gamma \in \mathbb{Z}$\ and $%
d>2n$. Let $\{t_{k}\}$ be a $q$-admissible weight sequence satisfying $%
\mathrm{\eqref{Asum1}}$ with $\sigma _{1}=\theta \left( \frac{q}{\theta }%
\right) ^{\prime }$ and $p=q$. Then%
\begin{equation*}
\big\|\lambda _{q,d}^{\ast }|\dot{f}_{\infty ,q}(\mathbb{R}%
^{n},\{t_{k-\gamma }\})\big\|\approx \big\|\lambda |\dot{f}_{\infty ,q}(%
\mathbb{R}^{n},\{t_{k-\gamma }\})\big\|.
\end{equation*}%
In addition if $\{t_{k}\}$ satisfying $\mathrm{\eqref{Asum2}}$ with $\sigma
_{2}\geq q$ and $\alpha _{2}\in \mathbb{R}$, then 
\begin{equation}
\big\|\lambda _{q,d}^{\ast }|\dot{f}_{\infty ,q}(\mathbb{R}%
^{n},\{t_{k-\gamma }\})\big\|\lesssim \big\|\lambda |\dot{f}_{\infty ,q}(%
\mathbb{R}^{n},\{t_{k}\})\big\|.  \label{Second}
\end{equation}
\end{lemma}

\begin{proof}
{The proof is similar to that of \cite[Lemma 2.3]{FJ90}. Obviously, }%
\begin{equation*}
\big\|\lambda |\dot{f}_{\infty ,q}(\mathbb{R}^{n},\{t_{k-\gamma }\})\big\|{%
\leq \big\|\lambda _{q,d}^{\ast }|\dot{f}_{\infty ,q}(\mathbb{R}%
^{n},\{t_{k-\gamma }\})\big\|.}
\end{equation*}%
{Let us prove the opposite inequality. For any }$j\in ${{$\mathbb{N}$}}$%
,m\in \mathbb{Z}^{n}${\ and any $k\in \mathbb{Z}${\ }we define }%
\begin{equation*}
{\Omega _{j,m}:=\{h\in \mathbb{Z}^{n}:2^{j-1}<\left\vert h-m\right\vert \leq
2^{j}\}\quad }\text{{and}}{\quad \Omega _{0,m}:=\{h\in \mathbb{Z}%
^{n}:\left\vert h-m\right\vert \leq 1\}.}
\end{equation*}%
{Let }$\frac{nq}{d-n}<\beta <q$. We observe that {for any $x\in Q_{k,m}$, }%
\begin{equation*}
{\sum_{h\in \mathbb{Z}^{n}}\frac{|\lambda _{k,h}|^{q}}{(1+|h-m|)^{d}}}
\end{equation*}%
{can be rewritten as}%
\begin{equation*}
\sum\limits_{j=0}^{\infty }\sum\limits_{h\in {\Omega _{j,m}}}\frac{%
\left\vert \lambda _{k,h}\right\vert ^{q}}{\left( 1+\left\vert
h-m\right\vert \right) ^{d}},
\end{equation*}%
which is bounded by {%
\begin{align}
& 2^{d}\sum\limits_{j=0}^{\infty }2^{-dj}\sum\limits_{h\in {\Omega _{j,m}}%
}\left\vert \lambda _{k,h}\right\vert ^{q}  \notag \\
\leq & 2^{d}\sum\limits_{j=0}^{\infty }2^{-dj}\Big(\sum\limits_{h\in {\Omega
_{j,m}}}\left\vert \lambda _{k,h}\right\vert ^{\beta }\Big)^{\frac{q}{\beta }%
}  \notag \\
=& 2^{d}\sum\limits_{j=0}^{\infty }2^{-dj}\Big(2^{kn}\int\limits_{\cup
_{z\in {\Omega _{j,m}}}Q_{k,z}}\sum\limits_{h\in \Omega _{j}}\left\vert
\lambda _{k,h}\right\vert ^{\beta }\chi _{k,h}(y)dy\Big)^{\frac{q}{\beta }}.
\label{key-est2}
\end{align}%
Let $x\in Q_{k,m}\subset P\in $}$\mathcal{Q}${\ and $y\in \cup _{z\in {%
\Omega _{j,m}}}Q_{k,z}$. Then $y\in Q_{k,z}$ for some $z\in $}${\Omega _{j,m}%
}${\ and }%
\begin{equation*}
{2^{j-1}<\left\vert z-m\right\vert \leq 2^{j}.}
\end{equation*}%
{Then }$\left\vert y-x\right\vert \lesssim 2^{j-k}$, which implies that $y$
is located in the ball $B(x,2^{j-k+\delta _{n}})$, $\delta _{n}\in \mathbb{N}
$. In addition, from the fact that%
\begin{align*}
\left\vert y-x_{P}\right\vert \leq & \left\vert y-x\right\vert +\left\vert
x-x_{P}\right\vert \\
\leq & 2^{j-k+\delta _{n}}+\sqrt{n}2^{-k_{P}}\leq 2^{j-k_{P}+c_{n}},\quad
c_{n}\in \mathbb{N},k_{P}=-\log _{2}l(P),k\geq k_{P}
\end{align*}%
we have that $y$ is located in the ball $B(x_{P},2^{j-k_{P}+c_{n}})$, where $%
x_{P}$ is the centre of $P$. Therefore, $\mathrm{\eqref{key-est2}}$ does not
exceed%
\begin{equation*}
c\text{ }\sum\limits_{j=0}^{\infty }2^{(\frac{nq}{\beta }-d)j}\Big(\mathcal{M%
}_{\beta }\big(\sum\limits_{h\in \mathbb{Z}^{n}}\lambda _{k,h}\chi
_{k,h}\chi _{B(x_{P},2^{j-k_{P}+c_{n}})}\big)(x)\Big)^{q}.
\end{equation*}%
Recall that%
\begin{equation*}
{\big\|\lambda _{q,d}^{\ast }|\dot{f}_{\infty ,q}(\mathbb{R}%
^{n},\{t_{k-\gamma }\})\big\|}^{q}=\sup_{P\in \mathcal{Q}}\frac{1}{|P|}%
\sum\limits_{k=-\log _{2}l(P)}^{\infty }\sum\limits_{m\in \mathbb{Z}^{n}}2^{k%
\frac{nq}{2}}{\big\|t{_{k-\gamma }}\lambda _{k,m,q,d}^{\ast }\chi
_{Q_{k,m}\cap P}|L_{q}(\mathbb{R}^{n})\big\|}^{q}.
\end{equation*}%
Using Lemma \ref{key-estimate1} and the fact that $d>\frac{nq}{\beta }+n$,
we obtain the desired estimate. To prove $\mathrm{\eqref{Second}}$ we use
again Lemma {\ref{key-estimate1} combined with Remark \ref{r-estimates}/(ii).%
}
\end{proof}

Let $\tilde{p}=p$ if $0<p<\infty $ and $\tilde{p}=q$ if $p=\infty $. For $%
\tilde{p}=q$, applying last lemma and {repeating the same arguments of \cite[%
Theorem 2.2]{FJ90}} we obtain the so called the $\varphi $-transform
characterization in the sense of Frazier and Jawerth, when for $0<p<\infty $%
, the proof of is given in {\cite{D20.1}. }It will play an important role in
the rest of the paper.

\begin{theorem}
\label{phi-tran1}Let $\alpha =(\alpha _{1},\alpha _{2})\in \mathbb{R}%
^{2},0<\theta \leq p\leq \infty $ and$\ 0<q<\infty $. Let $\{t_{k}\}\in \dot{%
X}_{\alpha ,\sigma ,\tilde{p}}$ be a $\tilde{p}$-admissible weight sequence
with $\sigma =(\sigma _{1}=\theta \left( \frac{\tilde{p}}{\theta }\right)
^{\prime },\sigma _{2}\geq \tilde{p})$. Let $\varphi $, $\psi \in \mathcal{S}%
(\mathbb{R}^{n})$ satisfying $\mathrm{\eqref{Ass1}}$\ through\ $\mathrm{%
\eqref{Ass3}}$. The operators 
\begin{equation*}
S_{\varphi }:\dot{F}_{p,q}(\mathbb{R}^{n},\{t_{k}\})\rightarrow \dot{f}%
_{p,q}(\mathbb{R}^{n},\{t_{k}\})
\end{equation*}%
and 
\begin{equation*}
T_{\psi }:\dot{f}_{p,q}(\mathbb{R}^{n},\{t_{k}\})\rightarrow \dot{F}_{p,q}(%
\mathbb{R}^{n},\{t_{k}\})
\end{equation*}%
are bounded. Furthermore, $T_{\psi }\circ S_{\varphi }$ is the identity on $%
\dot{F}_{p,q}(\mathbb{R}^{n},\{t_{k}\})$.
\end{theorem}

\begin{corollary}
Let $\alpha =(\alpha _{1},\alpha _{2})\in \mathbb{R}^{2},0<\theta \leq p\leq
\infty $ and$\ 0<q<\infty $. Let $\{t_{k}\}\in \dot{X}_{\alpha ,\sigma ,%
\tilde{p}}$ be a $\tilde{p}$-admissible weight sequence with $\sigma
=(\sigma _{1}=\theta \left( \frac{\tilde{p}}{\theta }\right) ^{\prime
},\sigma _{2}\geq \tilde{p})$. The definition of the spaces $\dot{F}_{p,q}(%
\mathbb{R}^{n},\{t_{k}\})$ is independent of the choices of $\varphi \in 
\mathcal{S}(\mathbb{R}^{n})$ satisfying $\mathrm{\eqref{Ass1}}$\ through\ $%
\mathrm{\eqref{Ass2}}$.
\end{corollary}

\begin{theorem}
Let $\alpha =(\alpha _{1},\alpha _{2})\in \mathbb{R}^{2},0<\theta \leq
p<\infty $ and $0<q<\infty $. Let $\{t_{k}\}\in \dot{X}_{\alpha ,\sigma ,p}$
be a $p$-admissible weight sequence with $\sigma =(\sigma _{1}=\theta \left( 
\frac{p}{\theta }\right) ^{\prime },\sigma _{2}\geq p)$. $\dot{F}_{p,q}(%
\mathbb{R}^{n},\{t_{k}\})$ are quasi-Banach spaces. They are Banach spaces
if $1\leq p<\infty $ and $1\leq q<\infty $.
\end{theorem}

We end this section with one more theorem, where the proof is given in {\cite%
{D20.1}.}

\begin{theorem}
\label{embeddings-S-inf}Let $0<\theta \leq p<\infty $ and $0<q<\infty $.
Let\ $\{t_{k}\}\in \dot{X}_{\alpha ,\sigma ,p}$ be a $p$-admissible weight
sequence with $\sigma =(\sigma _{1}=\theta \left( \frac{p}{\theta }\right)
^{\prime },\sigma _{2}\geq p)$ and $\alpha =(\alpha _{1},\alpha _{2})\in 
\mathbb{R}^{2}$.\ We have the embedding%
\begin{equation*}
\mathcal{S}_{\infty }(\mathbb{R}^{n})\hookrightarrow \dot{F}_{p,q}(\mathbb{R}%
^{n},\{t_{k}\})\hookrightarrow \mathcal{S}_{\infty }^{\prime }(\mathbb{R}%
^{n}).
\end{equation*}%
In addition $\mathcal{S}_{\infty }(\mathbb{R}^{n})$ is dense in $\dot{F}%
_{p,q}(\mathbb{R}^{n},\{t_{k}\})\mathrm{.}$
\end{theorem}

\section{Duality}

In this section we identify the duals of $\dot{F}_{p,q}(\mathbb{R}%
^{n},\{t_{k}\})$ spaces. The classical case, $\{t_{k}\}=\{2^{ks}\},s\in 
\mathbb{R}$, this was done in \cite[p. 176]{T1} and \cite[Sections 5 and 8]%
{FJ90}, while the anisotropic case is given in \cite{M08}.

We reduce the problem to corresponding sequence spaces. Before proving the
duality of these function spaces we present some results, which appeared in
the paper of Frazier and Jawerth \cite{FJ90} for classical Besov and
Triebel-Lizorkin spaces.

\begin{proposition}
\label{prop1}Let $0<\theta \leq q<\infty $. Let $\{t_{k}\}$ be a $q$%
-admissible weight sequence satisfying $\mathrm{\eqref{Asum1}}$ with $\sigma
_{1}=\theta \left( \frac{q}{\theta }\right) ^{\prime }$, $p=q$ and $j=k$.%
\textit{\ Suppose that for each dyadic cube }$Q_{k,m}$ there is a set $%
E_{Q_{k,m}}\subseteq Q_{k,m}$ with $|E_{Q_{k,m}}|>\varepsilon |Q_{k,m}|$, $%
\varepsilon >0$. Then%
\begin{equation*}
{\big\|}\lambda |\dot{f}_{\infty ,q}(\mathbb{R}^{n},\{t_{k}\}){\big\|}%
\approx \sup_{P\in \mathcal{Q},}\Big(\frac{1}{|P|}\int_{P}\sum\limits_{k=-%
\log _{2}l(P)}^{\infty }\sum\limits_{m\in \mathbb{Z}^{n}}2^{\frac{knq}{2}%
}t_{k}^{q}(x)|\lambda _{k,m}|^{q}\chi _{E_{Q_{k,m}}}(x)dx\Big)^{\frac{1}{q}}.
\end{equation*}
\end{proposition}

\begin{proof}
Since $\chi _{E_{Q}}\leq \chi _{Q}$ for all $Q\in \mathcal{Q}$, one the
direction is trivial. For the other, we use the estimate $\chi _{Q}\leq c$ $%
\mathcal{M}_{\varrho }(\chi _{E_{Q}\cap Q\cap P})$ for all $Q\subset P\in 
\mathcal{Q}$ with $0<\varrho <\min (1,\theta )$. Now Lemma \ref%
{key-estimate1} gives the desired estimate.
\end{proof}

\begin{remark}
Let $0<\theta \leq q<\infty $. Let $\{t_{k}\}$ be a $q$-admissible weight
sequence satisfying $\mathrm{\eqref{Asum1}}$ with $\sigma _{1}=\theta \left( 
\frac{q}{\theta }\right) ^{\prime }$, $p=q$ and $j=k$\textit{. }Suppose that
for each dyadic cube\textit{\ }$Q_{k,m}$ there is a set $E_{Q_{k,m}}%
\subseteq Q_{k,m}$ with $|E_{Q_{k,m}}|>\varepsilon |Q_{k,m}|$, $\varepsilon
>0$. Then%
\begin{equation*}
{\big\|}\lambda |\dot{f}_{\infty ,q}(\mathbb{R}^{n},\{t_{k}\}){\big\|}%
\lesssim \Big\|\Big(\sum\limits_{k=-\infty }^{\infty }\sum\limits_{m\in 
\mathbb{Z}^{n}}2^{\frac{knq}{2}}t_{k}^{q}|\lambda _{k,m}|^{q}\chi
_{E_{Q_{k,m}}}\Big)^{1/q}|L_{\infty }(\mathbb{R}^{n})\Big\|.
\end{equation*}
\end{remark}

For any dyadic cube $P$, we set 
\begin{equation*}
G_{P}^{q}(\lambda ,\{t_{k}\})(x):=\Big(\sum\limits_{k=-\log
_{2}l(P)}^{\infty }\sum\limits_{h\in \mathbb{Z}^{n}}2^{\frac{knq}{2}%
}t_{k}^{q}(x)|\lambda _{k,h}|^{q}\chi _{k,h}(x)\Big)^{1/q}.
\end{equation*}%
We put%
\begin{equation}
m_{P}^{q}(\lambda ,\{t_{k}\}):=\inf \Big\{\varepsilon :|\{x\in
P:G_{P}^{q}(\lambda ,\{t_{k}\})(x)>\varepsilon \}|<\frac{|P|}{4}\Big\}.
\label{inf-set}
\end{equation}%
We also set%
\begin{equation*}
m^{q}(\lambda ,\{t_{k}\})(x)=\sup_{P}m_{P}^{q}(\lambda ,\{t_{k}\})\chi
_{P}(x).
\end{equation*}%
Then we obtain.

\begin{proposition}
\label{prop2}Let $0<\theta \leq q<\infty $. Let $\{t_{k}\}$ be a $q$%
-admissible weight sequence satisfying $\mathrm{\eqref{Asum1}}$ with $\sigma
_{1}=\theta \left( \frac{q}{\theta }\right) ^{\prime }$, $p=q$ and $j=k$%
\textit{. Then}%
\begin{equation*}
{\big\|}\lambda |\dot{f}_{\infty ,q}(\mathbb{R}^{n},\{t_{k}\}){\big\|}%
\approx {\big\|}m^{q}(\lambda ,\{t_{k}\})|L_{\infty }(\mathbb{R}^{n}){\big\|}%
.
\end{equation*}
\end{proposition}

\begin{proof}
We use the arguments of \cite[Proposition 5.5]{FJ90}. Let $P$ be any dyadic
cube. We use the Chebyshev inequality, 
\begin{equation*}
|\{x\in P:G_{P}^{q}(\lambda ,\{t_{k}\})(x)>\varepsilon \}|
\end{equation*}%
is dominated by%
\begin{align*}
\frac{1}{\varepsilon ^{q}}\int_{P}(G_{P}^{q}(\lambda ,\{t_{k}\})(x))^{q}dx =&%
\frac{|P|}{\varepsilon ^{q}|P|}{\big\|}G_{P}^{q}(\lambda ,\{t_{k}\})\chi
_{P}|L_{q}(\mathbb{R}^{n}){\big\|}^{q} \\
\leq &\frac{|P|}{\varepsilon ^{q}}{\big\|}\lambda |\dot{f}_{\infty ,q}(%
\mathbb{R}^{n},\{t_{k}\}){\big\|}^{q}.
\end{align*}%
This term is less than to $\frac{|P|}{4}$ if $\varepsilon >4^{\frac{1}{q}}{%
\big\|}\lambda |\dot{f}_{\infty ,q}(\mathbb{R}^{n},\{t_{k}\}){\big\|}$.\
Hence, 
\begin{equation*}
{\big\|}m^{q}(\lambda ,\{t_{k}\})|L_{\infty }(\mathbb{R}^{n}){\big\|}\leq c{%
\big\|}\lambda |\dot{f}_{\infty ,q}(\mathbb{R}^{n},\{t_{k}\}){\big\|}.
\end{equation*}%
Now let 
\begin{align*}
& j(x) \\
=& \inf \Big\{j\in \mathbb{Z}:\Big(\sum_{k=j}^{\infty }\sum\limits_{h\in 
\mathbb{Z}^{n}}2^{\frac{knq}{2}}t_{k}^{q}(x)|\lambda _{k,h}|^{q}\chi
_{k,h}(x)\Big)^{1/q}\leq m^{q}(\lambda ,\{t_{k}\})(x)\Big\}.
\end{align*}%
and%
\begin{align*}
E_{Q_{k,h}}=& \left\{ x\in Q_{k,h}:2^{-j(x)}\geq l(Q_{k,h})\right\} \\
=& \Big\{x\in Q_{k,h}:G_{Q_{k,h}}^{q}(\lambda ,\{t_{k}\})(x)\leq
m^{q}(\lambda ,\{t_{k}\})(x)\Big\}
\end{align*}%
for any dyadic cube $Q_{k,h}$, $k\in \mathbb{Z}$ and $h\in \mathbb{Z}^{n}$.
By $\mathrm{\eqref{inf-set}}$, $|E_{Q_{k,h}}|\geq \frac{3|Q_{k,h}|}{4}$, and 
\begin{equation*}
\Big(\sum\limits_{k=-\infty }^{\infty }\sum\limits_{h\in \mathbb{Z}^{n}}2^{%
\frac{knq}{2}}t_{k}^{q}(x)|\lambda _{k,h}|^{q}\chi _{E_{Q_{k,h}}}(x)\Big)%
^{1/q}\leq c\text{ }m^{q}(\lambda ,\{t_{k}\})(x).
\end{equation*}%
From the last estimate and Proposition \ref{prop1}, we deduce that 
\begin{equation*}
{\big\|}\lambda |\dot{f}_{\infty ,q}(\mathbb{R}^{n},\{t_{k}\}){\big\|}%
\lesssim {\big\|}m^{q}(\lambda ,\{t_{k}\})|L_{\infty }(\mathbb{R}^{n}){\big\|%
}.
\end{equation*}
\end{proof}

\begin{remark}
Let $0<\theta \leq p<\infty ,0<q<\infty $. Let $\{t_{k}\}$ be a $p$%
-admissible weight sequence satisfying $\mathrm{\eqref{Asum1}}$ with $\sigma
_{1}=\theta \left( \frac{p}{\theta }\right) ^{\prime }$ and $j=k$. \textit{%
Then}%
\begin{equation*}
{\big\|}\lambda |\dot{f}_{p,q}(\mathbb{R}^{n},\{t_{k}\}){\big\|}\approx {%
\big\|}m^{q}(\lambda ,\{t_{k}\})|L_{p}(\mathbb{R}^{n}){\big\|}.
\end{equation*}
\end{remark}

By this proposition and Proposition \ref{prop1}, we obtain another
equivalent quasi-norm of $\dot{f}_{\infty ,q}(\mathbb{R}^{n},\{t_{k}\})$
spaces.

\begin{proposition}
\label{norm-equiv11}Let $0<\theta \leq q<\infty $. Let $\{t_{k}\}$ be a $q$%
-admissible weight sequence satisfying $\mathrm{\eqref{Asum1}}$ with $\sigma
_{1}=\theta \left( \frac{q}{\theta }\right) ^{\prime }$, $p=q$ and $j=k$%
\textit{. Then }$\lambda =\{\lambda _{k,m}\}_{k\in \mathbb{Z},m\in \mathbb{Z}%
^{n}}\in \dot{f}_{\infty ,q}(\mathbb{R}^{n},\{t_{k}\})$ \textit{if and only
if for each \ dyadic cube }$Q_{k,m}$ there is a subset $E_{Q_{k,m}}\subset
Q_{k,m}$ with $|E_{Q_{k,m}}|>\frac{|Q_{k,m}|}{2}$ (or any other, fixed,
number $0<\varepsilon <1$) such that%
\begin{equation*}
\Big\|\Big(\sum_{k=-\infty }^{\infty }\sum\limits_{m\in \mathbb{Z}^{n}}2^{%
\frac{knq}{2}}t_{k}^{q}(x)|\lambda _{k,m}|^{q}\chi _{E_{Q_{k,m}}}\Big)%
^{1/q}|L_{\infty }(\mathbb{R}^{n})\Big\|<\infty .
\end{equation*}%
Moreover, the infimum of this expression over all such collections $%
\{E_{Q_{k,m}}\}_{k\in \mathbb{Z},m\in \mathbb{Z}^{n}}$ is equivalent to ${%
\big\|}\lambda |\dot{f}_{\infty ,q}(\mathbb{R}^{n},\{t_{k}\}){\big\|}$.
\end{proposition}

Suppose that $1\leq p\leq \infty $. In the classical Lebesgue space,%
\begin{equation*}
{\big\|}f|L_{p}(\mathbb{R}^{n}){\big\|}=\sup \left\vert \int_{\mathbb{R}%
^{n}}f(x)g(x)dx\right\vert ,
\end{equation*}%
where the supremum is taken over all $g\in L_{p^{\prime }}(\mathbb{R}^{n})$
with ${\big\|}g|L_{p^{\prime }}(\mathbb{R}^{n}){\big\|}\leq 1$.

Our aim is to extend this result to $\dot{f}_{\infty ,q}(\mathbb{R}%
^{n},\{t_{k}\})$. Let $1<\theta \leq q<\infty $ and $\frac{1}{q}+\frac{1}{%
q^{\prime }}=1$. Let $\{t_{k}\}$ be a $q$-admissible weight sequence and let 
$\lambda =\{\lambda _{k,m}\}_{k\in \mathbb{Z},m\in \mathbb{Z}^{n}}\subset 
\mathbb{C}$. We define the conjugate norm to $\dot{f}_{\infty ,q}(\mathbb{R}%
^{n},\{t_{k}\})$ by 
\begin{equation*}
{\big\|}\lambda |\dot{f}_{\infty ,q}(\mathbb{R}^{n},\{t_{k}^{-1}\}){\big\|}%
^{\prime }=\sup_{\{s_{k,m}\}_{k\in \mathbb{Z},m\in \mathbb{Z}%
^{n}}}\sup_{P\in \mathcal{Q}}\Big|\int_{P}\frac{1}{|P|}\sum\limits_{k=-\log
_{2}l(P)}^{\infty }\sum\limits_{m\in \mathbb{Z}^{n}}\lambda
_{k,m}s_{k,m}\chi _{k,m}(x)dx\Big|,
\end{equation*}%
where the supremum is taken over all dyadic cube $P$\ and\ over all
sequence\ $s=\{s_{k,m}\}_{k\in \mathbb{Z},m\in \mathbb{Z}^{n}}\subset 
\mathbb{C}$ such that%
\begin{align*}
& {\big\|}s|\dot{f}_{\infty ,q^{\prime }}(\mathbb{R}^{n},\{2^{-nk}t_{k}^{-1}%
\}){\big\|}^{\star } \\
=& \sup_{P\in \mathcal{Q}}\Big(\frac{1}{|P|}\int_{P}\sum\limits_{k=-\log
_{2}l(P)}^{\infty }\sum\limits_{m\in \mathbb{Z}^{n}}2^{knq^{\prime }(\frac{1%
}{q^{\prime }}-\frac{1}{2})}(\tilde{t}_{k,m,q^{\prime }})^{q^{\prime
}}|s_{k,m}|^{q^{\prime }}\chi _{k,m}(x)dx\Big)^{\frac{1}{q^{\prime }}}\leq 1,
\end{align*}%
where $\tilde{t}_{k,m,q^{\prime }}=\big(\int_{Q_{k,m}}t_{k}^{-q^{\prime
}}(x)dx\big)^{\frac{1}{q^{\prime }}}$.

\begin{lemma}
\label{ball-daulity-norm}Let $0<\theta <\infty $ and $\max (\theta
,1)<q<\infty $. Let $\{t_{k}\}$ be a $q$-admissible weight sequence
satisfying $\mathrm{\eqref{Asum1}}$ with $\sigma _{1}=\theta \left( \frac{q}{%
\theta }\right) ^{\prime }$, $p=q$ and $j=k$\textit{. Let } $\lambda
=\{\lambda _{k,m}\}_{k\in \mathbb{Z},m\in \mathbb{Z}^{n}}\in \dot{f}_{\infty
,q}(\mathbb{R}^{n},\{t_{k}\})$. Then%
\begin{equation*}
{\big\|}\lambda |\dot{f}_{\infty ,q}(\mathbb{R}^{n},\{t_{k}\}){\big\|}%
\approx {\big\|}\lambda |\dot{f}_{\infty ,q}(\mathbb{R}^{n},\{t_{k}^{-1}\}){%
\big\|}^{\prime }.
\end{equation*}
\end{lemma}

\begin{proof}
Let\ $s=\{s_{k,m}\}_{k\in \mathbb{Z},m\in \mathbb{Z}^{n}}\subset \mathbb{C}$
be such that ${\big\|}s|\dot{f}_{\infty ,q^{\prime }}(\mathbb{R}%
^{n},\{2^{-nk}t_{k}^{-1}\}){\big\|}^{\star }\leq 1$. Since $\frac{1}{q}+%
\frac{1}{q^{\prime }}=1$, by H\"{o}lder's inequality%
\begin{align*}
& \frac{1}{|P|}\Big|\int_{P}\sum\limits_{k=-\log _{2}l(P)}^{\infty
}\sum\limits_{m\in \mathbb{Z}^{n}}\lambda _{k,m}s_{k,m}\chi _{k,m}(x)dx\Big|
\\
\leq & {\big\|}\lambda |\dot{f}_{\infty ,q}(\mathbb{R}^{n},\{t_{k}\}){\big\|%
\big\|}s|f_{\infty ,q^{\prime }}(\mathbb{R}^{n},\{2^{-nk}t_{k}^{-1}\}){\big\|%
}
\end{align*}%
for any dyadic cube $P$. Using $\mathrm{\eqref{equi -f-inf-sequence}}$, we
derive that%
\begin{equation*}
{\big\|}s|f_{\infty ,q^{\prime }}(\mathbb{R}^{n},\{2^{-nk}t_{k}^{-1}\}){%
\big\|=\big\|}s|f_{\infty ,q}(\mathbb{R}^{n},\{2^{-nk}t_{k}\}){\big\|}%
^{\star }
\end{equation*}%
and 
\begin{equation*}
{\big\|}\lambda |\dot{f}_{\infty ,q^{\prime }}(\mathbb{R}^{n},\{t_{k}^{-1}\})%
{\big\|}^{\prime }\leq {\big\|}\lambda |\dot{f}_{\infty ,q^{\prime }}(%
\mathbb{R}^{n},\{t_{k}^{-1}\}){\big\|}.
\end{equation*}%
Let us prove that the converse holds. Let $s=\{s_{k,m}\}_{k\in \mathbb{Z}%
,m\in \mathbb{Z}^{n}}\subset \mathbb{C}$\ be a sequence defined by%
\begin{equation*}
s_{k,m}=(t_{k,m,q})^{q-1}2^{kn(\frac{1}{2}+\frac{q}{2q^{\prime }})}(\tilde{t}%
_{k,m,q^{\prime }})^{-1}\Big|\frac{\lambda _{k,m}}{{\big\|}\lambda |\dot{f}%
_{\infty ,q}(\mathbb{R}^{n},\{t_{k}\}){\big\|}}\Big|^{q-1}\text{sgn }\lambda
_{k,m}.
\end{equation*}%
We let the reader to check that 
\begin{equation*}
{\big\|}s|\dot{f}_{\infty ,q^{\prime }}(\mathbb{R}^{n},\{2^{-nk}t_{k}^{-1}\})%
{\big\|}^{\star }=1.
\end{equation*}%
Since $t_{k}^{q}\in A_{q}(\mathbb{R}^{n}),k\in \mathbb{Z}$, we have 
\begin{equation*}
|Q_{k,m}|^{-1}t_{k,m,q}\tilde{t}_{k,m,q^{\prime }}\leq c,\quad k\in \mathbb{Z%
},m\in \mathbb{Z}^{n},
\end{equation*}%
where the positive constant $c$ is independent of $k$ and $m$. Consequently,
we obtain 
\begin{align*}
{\big\|}\lambda |\dot{f}_{\infty ,q}(\mathbb{R}^{n},\{t_{k}^{-1}\}){\big\|}%
^{\prime }\geq & \sup_{P\in \mathcal{Q}}\Big|\int_{P}\frac{1}{|P|}%
\sum\limits_{k=-\log _{2}l(P)}^{\infty }\sum\limits_{m\in \mathbb{Z}%
^{n}}\lambda _{k,m}s_{k,m}\chi _{k,m}(x)dx\Big| \\
=\frac{1}{c}& {\big\|}\lambda |\dot{f}_{\infty ,q}(\mathbb{R}^{n},\{t_{k}\}){%
\big\|}
\end{align*}%
and hence we complete the proof of lemma.
\end{proof}

To prove the main result of this section, we need the following result.

\begin{theorem}
\label{main-result-duality1}Let $0<\theta <1<q<\infty $. Let $\{t_{k}\}$ be
a $1$-admissible weight sequence satisfying $\mathrm{\eqref{Asum1}}$ with $%
\sigma _{1}=\theta \left( \frac{1}{\theta }\right) ^{\prime }$, $p=1$ and $%
j=k$\textit{. Assume that }$\{t_{k}^{-1}\}$ is a $q^{\prime }$-admissible
weight sequence satisfying $\mathrm{\eqref{Asum1}}$ with $\sigma _{1}=\theta
\left( \frac{q^{\prime }}{\theta }\right) ^{\prime }$, $p=q^{\prime }$ and $%
j=k$. \textit{Then }%
\begin{equation*}
\big(\dot{f}_{1,q}(\mathbb{R}^{n},\{t_{k}\})\big)^{\ast }=\dot{f}_{\infty
,q^{\prime }}(\mathbb{R}^{n},\{t_{k}^{-1}\}).
\end{equation*}%
In particular, if $\lambda =\{\lambda _{k,m}\}_{k\in \mathbb{Z},m\in \mathbb{%
Z}^{n}}\in f_{\infty ,q^{\prime }}(\mathbb{R}^{n},\{t_{k}^{-1}\})$, then the
map 
\begin{equation*}
s=\{s_{k,m}\}_{k\in \mathbb{Z},m\in \mathbb{Z}^{n}}\rightarrow l_{\lambda
}(s)=\sum_{k=-\infty }^{\infty }\sum_{m\in \mathbb{Z}^{n}}s_{k,m}\bar{\lambda%
}_{k,m}
\end{equation*}%
defined a continuos linear functional on $\dot{f}_{1,q}(\mathbb{R}%
^{n},\{t_{k}\})$ with 
\begin{equation*}
{\big\|}l_{\lambda }|\big(\dot{f}_{1,q}(\mathbb{R}^{n},\{t_{k}\})\big)^{\ast
}{\big\|}\approx {\big\|}\lambda |\dot{f}_{\infty ,q^{\prime }}(\mathbb{R}%
^{n},\{t_{k}^{-1}\}){\big\|},
\end{equation*}%
and every $l\in \big(\dot{f}_{1,q}(\mathbb{R}^{n},\{t_{k}\})\big)^{\ast }$
is of this form for some $\lambda \in \dot{f}_{\infty ,q^{\prime }}(\mathbb{R%
}^{n},\{t_{k}^{-1}\})$.
\end{theorem}

\begin{proof}
We will use the idea from \cite[Theorem 5.9]{FJ90}. Let $Q_{k,h}$ be a
dyadic cube, $k\in \mathbb{Z},h\in \mathbb{Z}^{n}$. We set 
\begin{equation*}
E_{Q_{k,h}}=\Big\{x\in Q_{k,h}:G_{Q_{k,h}}^{q^{\prime }}(\lambda
,\{t_{k}^{-1}\})(x)\leq m^{q^{\prime }}(\lambda ,\{t_{k}^{-1}\})(x)\Big\}.
\end{equation*}%
Then $|E_{Q_{k,h}}|\geq \frac{3|Q_{k,h}|}{4}$ and%
\begin{equation*}
|s_{k,m}||\lambda _{k,h}|=\frac{1}{|E_{Q_{k,h}}|}%
\int_{E_{Q_{k,h}}}|s_{k,h}||\lambda _{k,h}|dx\leq \frac{4}{3|Q_{k,h}|}%
\int_{E_{Q_{k,h}}}|s_{k,h}||\lambda _{k,h}|dx.
\end{equation*}%
Using the H\"{o}lder inequality, we obtain 
\begin{align*}
|l_{\lambda }(s)|\leq & \frac{4}{3}\int_{\mathbb{R}^{n}}\sum_{k=-\infty
}^{\infty }\sum_{h\in \mathbb{Z}^{n}}2^{\frac{kn}{2}%
}t_{k}(x)|s_{k,h}|t_{k}^{-1}(x)2^{\frac{kn}{2}}|\lambda _{k,h}|\chi
_{E_{Q_{k,h}}}(x)dx \\
\leq & \frac{4}{3}\int_{\mathbb{R}^{n}}\Big(\sum_{k=-\infty }^{\infty
}\sum_{h\in \mathbb{Z}^{n}}2^{\frac{kn}{2}q}t_{k}^{q}(x)|s_{k,h}|^{q}\chi
_{E_{Q_{k,h}}}(x)\Big)^{1/q} \\
& \Big(\sum_{k=-\infty }^{\infty }\sum_{h\in \mathbb{Z}^{n}}2^{\frac{kn}{2}%
q^{\prime }}t_{k}^{-q^{\prime }}(x)|\lambda _{k,h}|^{q^{\prime }}\chi
_{E_{Q_{k,h}}}(x)\Big)^{1/q^{\prime }}dx.
\end{align*}%
The last term is bounded by 
\begin{align*}
& c{\big\|}s|\dot{f}_{1,q}(\mathbb{R}^{n},\{t_{k}\}){\big\|}\Big\|\Big(%
\sum_{k=-\infty }^{\infty }\sum_{h\in \mathbb{Z}^{n}}2^{\frac{kn}{2}%
q^{\prime }}t_{k}^{-q^{\prime }}|\lambda _{k,h}|^{q^{\prime }}\chi
_{E_{Q_{k,h}}}\Big)^{1/q^{\prime }}|L_{\infty }(\mathbb{R}^{n})\Big\| \\
\lesssim & {\big\|}s|\dot{f}_{1,q}(\mathbb{R}^{n},\{t_{k}\}){\big\|\big\|}%
m^{q^{\prime }}(\lambda ,\{t_{k}^{-1}\}){\big\|}_{\infty } \\
\lesssim & {\big\|}s|\dot{f}_{1,q}(\mathbb{R}^{n},\{t_{k}\}){\big\|\big\|}%
\lambda |\dot{f}_{\infty ,q^{\prime }}(\mathbb{R}^{n},\{t_{k}^{-1}\}){\big\|}%
,
\end{align*}%
by Proposition \ref{prop2}. Therefore, 
\begin{equation*}
{\big\|}l_{\lambda }|\big(\dot{f}_{1,q}(\mathbb{R}^{n},\{t_{k}\}\big)^{\ast }%
{\big\|}\lesssim {\big\|}\lambda |\dot{f}_{\infty ,q^{\prime }}(\mathbb{R}%
^{n},\{t_{k}^{-1}\}){\big\|}.
\end{equation*}%
Clearly every $l\in \big(\dot{f}_{1,q}(\mathbb{R}^{n},\{t_{k}\})\big)^{\ast }
$ is of the form $s\mapsto \sum_{k=-\infty }^{\infty }s_{k,h}\bar{\lambda}%
_{k,h}$ for some $\lambda =\{\lambda _{k,h}\}_{k\in \mathbb{Z},h\in \mathbb{Z%
}^{n}}$. Now, the norm 
\begin{equation*}
{\big\|}\lambda |\dot{f}_{\infty ,q^{\prime }}(\mathbb{R}^{n},\{t_{k}^{-1}\})%
{\big\|,}
\end{equation*}%
is equivalent to%
\begin{equation}
\sup \Big|\int_{P}\frac{1}{|P|}\sum\limits_{k=-\log _{2}l(P)}^{\infty
}\sum\limits_{h\in \mathbb{Z}^{n}}\lambda _{k,h}\kappa _{k,h}\chi _{k,h}(x)dx%
\Big|,  \label{main-est}
\end{equation}%
where the supremum is taking all dyadic cube $P$ and over all sequence of $%
\kappa =\{\kappa _{k,h}\}_{k\in \mathbb{Z},h\in \mathbb{Z}^{n}}$ such that 
\begin{equation*}
{\big\|}\kappa |\dot{f}_{\infty ,q}(\mathbb{R}^{n},\{2^{-nk}t_{k}\}{\big\|}%
^{\ast }\leq 1,
\end{equation*}%
see Lemma \ref{ball-daulity-norm}. Let $D_{P}=\{D_{k,h,P}\}_{k\mathbb{\in Z}%
,h\in \mathbb{Z}^{n}}$ where 
\begin{equation*}
D_{k,h,P}=\left\{ 
\begin{array}{ccc}
0 & \text{if} & k<k_{P}, \\ 
0 & \text{if} & k\geq k_{P}\text{ and }Q_{k,h}\cap P=\emptyset , \\ 
\int_{P}\frac{|\kappa _{k,h}|}{\left\vert P\right\vert }\chi _{k,h}(x)dx & 
\text{if} & k\geq k_{P}\text{ and }Q_{k,h}\subset P,%
\end{array}%
\right. 
\end{equation*}%
and $k_{P}:=-\log _{2}l(P)$. The integral in $\mathrm{\eqref{main-est}}$ is
just%
\begin{equation*}
\sum\limits_{k=k_{P}}^{\infty }\sum\limits_{h\in \mathbb{Z}^{n}}|\lambda
_{k,h}|\int_{P}\frac{1}{|P|}|\kappa _{k,h}|\chi
_{k,h}(x)dx=\sum\limits_{k=k_{P}}^{\infty }\sum\limits_{h\in \mathbb{Z}%
^{n}}|\lambda _{k,h}|D_{k,h,P},
\end{equation*}%
which can be estimated by%
\begin{equation*}
{\big\|}l|\big(\dot{f}_{1,q}(\mathbb{R}^{n},\{t_{k}\})\big)^{\ast }{\big\|%
\big\|}D_{P}|\dot{f}_{1,q}(\mathbb{R}^{n},\{t_{k}\}){\big\|},
\end{equation*}%
provided that 
\begin{equation}
{\big\|}D_{P}|\dot{f}_{1,q}(\mathbb{R}^{n},\{t_{k}\}){\big\|}\lesssim 1.
\label{claim}
\end{equation}%
The claim $\mathrm{\eqref{claim}}$ can be reformulated as showing that%
\begin{equation}
\int_{P}\Big(\sum\limits_{k=k_{P}}^{\infty }\sum\limits_{h\in \mathbb{Z}%
^{n}}2^{knq(\frac{1}{2}+1)}t_{k,h,1}^{q}D_{k,h,P}^{q}\chi _{k,h}(x)\Big)%
^{1/q}dx\lesssim 1.  \label{equa-1}
\end{equation}%
Obviously, by H\"{o}lder's inequality we get%
\begin{equation*}
t_{k,h,1}\leq |Q_{k,h}|^{\frac{1}{q^{\prime }}}t_{k,h,q},\quad k\mathbb{\in Z%
},h\in \mathbb{Z}^{n},
\end{equation*}%
which implies that%
\begin{equation*}
2^{kn(\frac{1}{2}+1)}t_{k,h,1}D_{k,h,P}\leq 2^{kn(\frac{1}{q}-\frac{1}{2})}%
\frac{t_{k,h,q}|\kappa _{k,h}|}{\left\vert P\right\vert },\quad
Q_{k,h}\subset P.
\end{equation*}%
Therefore, the left-hand side of $\mathrm{\eqref{equa-1}}$ is bounded by%
\begin{align*}
& c\int_{P}\Big(\sum\limits_{k=k_{P}}^{\infty }\sum\limits_{h\in \mathbb{Z}%
^{n}}2^{kn(\frac{1}{q}-\frac{1}{2})q}\frac{t_{k,h,q}^{q}|\kappa _{k,h}|^{q}}{%
\left\vert P\right\vert ^{q}}\chi _{k,h}\Big)^{1/q}dy \\
\lesssim & \frac{1}{\left\vert P\right\vert ^{\frac{1}{q}}}\Big\|\Big(%
\sum\limits_{k=k_{P}}^{\infty }\sum\limits_{h\in \mathbb{Z}^{n}}2^{kn(\frac{1%
}{q}-\frac{1}{2})q}t_{k,h,q}^{q}|\kappa _{k,h}|^{q}\chi _{k,h}\Big)%
^{1/q}\chi _{P}|L_{q}(\mathbb{R}^{n})\Big\| \\
\lesssim & {\big\|}s|\dot{f}_{\infty ,q}(\mathbb{R}^{n},\{2^{-nk}t_{k}\}{%
\big\|}^{\ast } \\
\lesssim & 1,
\end{align*}%
where the first estimate follows by H\"{o}lder's inequality. Consequently, 
\begin{equation*}
{\big\|}\lambda |\dot{f}_{\infty ,q^{\prime }}(\mathbb{R}^{n},\{2^{-nk}t_{k}%
\}{\big\|}\lesssim {\big\|}l|\dot{f}_{1,q}(\mathbb{R}^{n},\{t_{k}\}){\big\|}
\end{equation*}%
and hence completes the proof of this theorem.
\end{proof}

Using the notation introduced above, we may now state the main result of
this section.

\begin{theorem}
\label{main-result-duality2}Let $0<\theta <1<q<\infty $. Let $\{t_{k}\}\in 
\dot{X}_{\alpha ,\sigma ,1}$ be a $1$-admissible weight sequence with $%
\sigma =(\sigma _{1}=\theta \left( \frac{1}{\theta }\right) ^{\prime
},\sigma _{2}\geq 1)$\textit{. Assume that }$\{t_{k}^{-1}\}\in \dot{X}%
_{\alpha ,\sigma ,q^{\prime }}$ is a $q^{\prime }$-admissible weight
sequence with $\sigma =(\sigma _{1}=\theta \left( \frac{q^{\prime }}{\theta }%
\right) ^{\prime },\sigma _{2}\geq q^{\prime })$. \textit{Then }%
\begin{equation*}
\big(\dot{F}_{1,q}(\mathbb{R}^{n},\{t_{k}\})\big)^{\ast }=\dot{F}_{\infty
,q^{\prime }}(\mathbb{R}^{n},\{t_{k}^{-1}\}).
\end{equation*}%
In particular, if $g\in \dot{F}_{\infty ,q^{\prime }}(\mathbb{R}%
^{n},\{t_{k}^{-1}\})$, then the map, given by $l_{g}(f)=\langle f,g\rangle ,$
defined initially for $f\in \mathcal{S}_{\infty }(\mathbb{R}^{n})$ extends
to a continuous linear functional on $\dot{F}_{1,q}(\mathbb{R}%
^{n},\{t_{k}\}) $ with 
\begin{equation*}
{\big\|}g|\dot{F}_{\infty ,q^{\prime }}(\mathbb{R}^{n},\{t_{k}^{-1}\}){\big\|%
}\approx {\big\|}l_{g}|(\dot{F}_{1,q}(\mathbb{R}^{n},\{t_{k}\}))^{\ast }{%
\big\|}
\end{equation*}%
and every $l\in \big(\dot{F}_{1,q}(\mathbb{R}^{n},\{t_{k}\})\big)^{\ast }$
satisfies $l=l_{g}$ for some $g\in \dot{F}_{\infty ,q^{\prime }}(\mathbb{R}%
^{n},\{t_{k}^{-1}\})$.
\end{theorem}

\begin{proof}
We follow the arguments of \cite[Theorem 5.13]{FJ90}. We may choose $\psi
=\varphi $ satisfies \eqref{Ass1}-\eqref{Ass3}. By Lemma \ref{DW-lemma1} and
Theorem \ref{main-result-duality1} we have that for any $f\in \mathcal{S}%
_{\infty }(\mathbb{R}^{n})$ and $g\in \dot{F}_{\infty ,q^{\prime }}(\mathbb{R%
}^{n},\{t_{k}^{-1}\}),$%
\begin{equation*}
|l_{g}(f)|=|\langle f,g\rangle |=|\langle S_{\varphi }f,S_{\varphi }g\rangle
|\leq {\big\|}S_{\varphi }g|\dot{f}_{\infty ,q^{\prime }}(\mathbb{R}%
^{n},\{t_{k}^{-1}\}){\big\|\big\|}S_{\varphi }f|\dot{f}_{1,q}(\mathbb{R}%
^{n},\{t_{k}\}){\big\|}.
\end{equation*}%
By Theorem \ref{phi-tran1} we have that%
\begin{equation*}
|l_{g}(f)|\lesssim {\big\|}g|\dot{F}_{\infty ,q^{\prime }}(\mathbb{R}%
^{n},\{t_{k}^{-1}\}){\big\|\big\|}f|\dot{F}_{1,q}(\mathbb{R}^{n},\{t_{k}\}){%
\big\|}.
\end{equation*}%
Conversely, suppose that $l\in \big(\dot{F}_{1,q}(\mathbb{R}^{n},\{t_{k}\})%
\big)^{\ast }$. Then%
\begin{equation*}
l_{1}=l\circ T_{\varphi }\in \big(\dot{f}_{1,q}(\mathbb{R}^{n},\{t_{k}\})%
\big)^{\ast }.
\end{equation*}%
Thanks to Theorem \ref{main-result-duality1} there exists $\lambda
=\{\lambda _{k,m}\}_{k\in \mathbb{Z},m\in \mathbb{Z}^{n}}\in \dot{f}_{\infty
,q^{\prime }}(\mathbb{R}^{n},\{t_{k}^{-1}\})$ such that%
\begin{equation*}
l_{1}(s)=\langle s,\lambda \rangle ,\quad s\in \dot{f}_{1,q}(\mathbb{R}%
^{n},\{t_{k}\})
\end{equation*}%
and%
\begin{equation*}
{\big\|}l_{1}|\big(\dot{f}_{1,q}(\mathbb{R}^{n},\{t_{k}\})\big)^{\ast }{%
\big\|}\approx {\big\|}\lambda |\dot{f}_{\infty ,q^{\prime }}(\mathbb{R}%
^{n},\{t_{k}^{-1}\}){\big\|}.
\end{equation*}%
By Theorem \ref{phi-tran1} we have that%
\begin{align*}
{\big\|}T_{\psi }\lambda |\dot{F}_{\infty ,q^{\prime }}(\mathbb{R}%
^{n},\{t_{k}^{-1}\}){\big\|}{\lesssim \big\|}& \lambda |\dot{f}_{\infty
,q^{\prime }}(\mathbb{R}^{n},\{t_{k}^{-1}\}){\big\|} \\
\lesssim & {\big\|}l_{1}|\big(\dot{f}_{1,q}(\mathbb{R}^{n},\{t_{k}\})\big)%
^{\ast }{\big\|.}
\end{align*}%
Finally, for any $f\in \mathcal{S}_{\infty }(\mathbb{R}^{n})$%
\begin{equation*}
l(f)=l(T_{\varphi }S_{\varphi }(f))=l\circ T_{\varphi }(S_{\varphi
}(f))=l_{1}(S_{\varphi }(f))=\langle S_{\varphi }(f),\lambda \rangle
=\langle f,T_{\psi }\lambda \rangle .
\end{equation*}%
This completes the proof.
\end{proof}

The goal of the rest of this section is to identify the duals of $\dot{F}%
_{p,q}(\mathbb{R}^{n},\{t_{k}\})$ spaces for $1<p<\infty $ and $1<q<\infty $%
. Again, This case was established by working on the sequence space $\dot{f}%
_{p,q}(\mathbb{R}^{n},\{t_{k}\})$.

\begin{theorem}
\label{main-result-duality3}Let $1<\theta \leq p<\infty $ and $1<q<\infty $.
Let\ $\{t_{k}\}$ be a $p$-admissible weight sequence satisfying $\mathrm{%
\eqref{Asum1}}$ with $\sigma _{1}=\theta \left( \frac{p}{\theta }\right)
^{\prime }$ and $j=k$\textit{. Then }%
\begin{equation*}
\big(\dot{f}_{p,q}(\mathbb{R}^{n},\{t_{k}\})\big)^{\ast }=\dot{f}_{p^{\prime
},q^{\prime }}(\mathbb{R}^{n},\{t_{k}^{-1}\}).
\end{equation*}%
In particular, if $\lambda =\{\lambda _{k,m}\}_{k\in \mathbb{Z},m\in \mathbb{%
Z}^{n}}\in \dot{f}_{p^{\prime },q^{\prime }}(\mathbb{R}^{n},\{t_{k}^{-1}\})$%
, then the map 
\begin{equation*}
s=\{s_{k,m}\}_{k\in \mathbb{Z},m\in \mathbb{Z}^{n}}\rightarrow l_{\lambda
}(s)=\sum_{k=-\infty }^{\infty }\sum_{m\in \mathbb{Z}^{n}}s_{k,m}\bar{\lambda%
}_{k,m}
\end{equation*}%
defined a continuos linear functional on $\dot{f}_{p,q}(\mathbb{R}%
^{n},\{t_{k}\})$ with 
\begin{equation*}
{\big\|}l_{\lambda }|\big(\dot{f}_{p,q}(\mathbb{R}^{n},\{t_{k}\})\big)^{\ast
}{\big\|}\approx {\big\|}\lambda |\dot{f}_{p^{\prime },q^{\prime }}(\mathbb{R%
}^{n},\{t_{k}^{-1}\}){\big\|},
\end{equation*}%
and every $l\in \big(\dot{f}_{p,q}(\mathbb{R}^{n},\{t_{k}\})\big)^{\ast }$
is of this form for some $\lambda \in \dot{f}_{p^{\prime },q^{\prime }}(%
\mathbb{R}^{n},\{t_{k}^{-1}\})$.
\end{theorem}

\begin{proof}
We follow the arguments of \cite[Remark 5.14]{FJ90} and \cite[Theorem 4.2]%
{M08}. Let $s\in \dot{f}_{p,q}(\mathbb{R}^{n},\{t_{k}\})$. We have%
\begin{equation*}
\sum_{k=-\infty }^{\infty }\sum_{m\in \mathbb{Z}^{n}}|s_{k,m}||\bar{\lambda}%
_{k,m}|=\sum_{k=-\infty }^{\infty }\sum_{m\in \mathbb{Z}^{n}}%
\int_{Q_{k,m}}2^{\frac{kn}{2}}t_{k}(x)|s_{k,m}|t_{k}^{-1}(x)2^{\frac{kn}{2}%
}|\lambda _{k,m}|dx.
\end{equation*}%
H\"{o}lder's inequality yields that $|l_{\lambda }(s)|$ can be estimated by%
\begin{align*}
& \int_{\mathbb{R}^{n}}\Big(\sum_{k=-\infty }^{\infty }\sum_{m\in \mathbb{Z}%
^{n}}2^{\frac{kn}{2}q}t_{k}^{q}(x)|s_{k,m}|^{q}\chi _{k,m}(x)\Big)^{1/q}\Big(%
\sum_{k=-\infty }^{\infty }\sum_{m\in \mathbb{Z}^{n}}2^{\frac{kn}{2}%
q^{\prime }}t_{k}^{-q^{\prime }}(x)|\lambda _{k,m}|^{q^{\prime }}\chi
_{k,m}(x)\Big)^{1/q^{\prime }}dx \\
& \leq {\big\|}s|\dot{f}_{p,q}(\mathbb{R}^{n},\{t_{k}\}){\big\|\big\|}s|\dot{%
f}_{p^{\prime },q^{\prime }}(\mathbb{R}^{n},\{t_{k}^{-1}\}){\big\|},
\end{align*}%
which yields 
\begin{equation*}
{\big\|}l_{\lambda }|\big(\dot{f}_{p,q}(\mathbb{R}^{n},\{t_{k}\})\big)^{\ast
}{\big\|}\leq {\big\|}\lambda |\dot{f}_{p^{\prime },q^{\prime }}(\mathbb{R}%
^{n},\{t_{k}^{-1}\}){\big\|}.
\end{equation*}%
Let $I:\dot{f}_{p,q}(\mathbb{R}^{n},\{t_{k}\})\rightarrow L_{p}(\ell _{q})$
be given by%
\begin{equation*}
I(s)=\Big\{\sum_{m\in \mathbb{Z}^{n}}2^{\frac{kn}{2}}t_{k}s_{k,m}\chi _{k,m}%
\Big\}_{k\in \mathbb{Z}}.
\end{equation*}%
Then $I$ is an isometry. Let $l\in \big(\dot{f}_{p,q}(\mathbb{R}%
^{n},\{t_{k}\})\big)^{\ast }$. By the Hahn-Banach Theorem, there exists $%
\tilde{l}\in \big(L_{p}(\ell _{q})\big)^{\ast }$ such $\tilde{l}\circ I=l$
and 
\begin{equation*}
{\big\|}\tilde{l}|\big(L_{p}(\ell _{q})\big)^{\ast }{\big\|=\big\|}l|\big(%
\dot{f}_{p,q}(\mathbb{R}^{n},\{t_{k}\})\big)^{\ast }{\big\|.}
\end{equation*}%
By Proposition 2.11.1 in \cite{T1}, 
\begin{equation*}
\tilde{l}(f)=\langle f,g\rangle =\int_{\mathbb{R}^{n}}\sum_{j=-\infty
}^{\infty }f_{j}(x)g_{j}(x)dx
\end{equation*}%
for some $g=\{g_{j}\}_{j\in \mathbb{Z}}\in L_{p^{\prime }}(\ell _{q^{\prime
}})$. Let $s\in \dot{f}_{p,q}(\mathbb{R}^{n},\{t_{k}\})$. Then%
\begin{align*}
l(s)=& \tilde{l}\circ I(s)=\tilde{l}(I(s)) \\
=& \langle I(s),g\rangle \\
=& \int_{\mathbb{R}^{n}}\sum_{j=-\infty }^{\infty }\sum_{m\in \mathbb{Z}%
^{n}}2^{\frac{jn}{2}}t_{j}(x)s_{j,m}\chi _{j,m}(x)g_{j}(x)dx \\
=& \sum_{j=-\infty }^{\infty }\sum_{m\in \mathbb{Z}^{n}}s_{j,m}\lambda _{j,m}
\end{align*}%
where%
\begin{equation*}
\lambda _{j,m}=2^{\frac{jn}{2}}\int_{Q_{j,m}}t_{j}(x)g_{j}(x)dx,\quad j\in 
\mathbb{Z},m\in \mathbb{Z}^{n}.
\end{equation*}%
We have%
\begin{equation*}
|\lambda _{j,m}|\leq 2^{-\frac{jn}{2}}\mathcal{M}(t_{j}g_{j}),\quad j\in 
\mathbb{Z},m\in \mathbb{Z}^{n}.
\end{equation*}%
Since $t_{k}^{p}\in A_{\frac{p}{\theta }}(\mathbb{R}^{n})\subset A_{p}(%
\mathbb{R}^{n}),k\in \mathbb{Z},$ it follows by Lemma \ref{Ap-Property}/(ii)
that $t_{k}^{-p^{\prime }}\in A_{p^{\prime }}(\mathbb{R}^{n}),k\in \mathbb{Z}
$. By Lemma \ref{FS-inequality}, we derive%
\begin{equation*}
{\big\|}\lambda |\dot{f}_{p^{\prime },q^{\prime }}(\mathbb{R}%
^{n},\{t_{k}^{-1}\}){\big\|\lesssim \big\|}g|L_{p^{\prime }}(\ell
_{q^{\prime }}){\big\|\lesssim \big\|}l|\big(\dot{f}_{p,q}(\mathbb{R}%
^{n},\{t_{k}\})\big)^{\ast }{\big\|,}
\end{equation*}%
which completes the proof of Theorem \ref{main-result-duality3}.
\end{proof}

Similarly as in Theorem \ref{main-result-duality2} we obtain.

\begin{theorem}
\label{main-result-duality2.1}Let $1<\theta \leq p<\infty $ and $1<q<\infty $%
. Let $\{t_{k}\}\in \dot{X}_{\alpha ,\sigma ,p}$ be a $p$-admissible weight
sequence with $\sigma =(\sigma _{1}=\theta \left( \frac{p}{\theta }\right)
^{\prime },\sigma _{2}\geq p)$.\textit{\ Then }%
\begin{equation*}
\big(\dot{F}_{p,q}(\mathbb{R}^{n},\{t_{k}\})\big)^{\ast }=\dot{F}_{p^{\prime
},q^{\prime }}(\mathbb{R}^{n},\{t_{k}^{-1}\}).
\end{equation*}%
In particular, if $g\in \dot{F}_{p^{\prime },q^{\prime }}(\mathbb{R}%
^{n},\{t_{k}^{-1}\})$, then the map, given by $l_{g}(f)=\langle f,g\rangle ,$
defined initially for $f\in \mathcal{S}_{\infty }(\mathbb{R}^{n})$ extends
to a continuous linear functional on $\dot{F}_{p,q}(\mathbb{R}%
^{n},\{t_{k}\}) $ with 
\begin{equation*}
{\big\|}g|\dot{F}_{p^{\prime },q^{\prime }}(\mathbb{R}^{n},\{t_{k}^{-1}\}){%
\big\|}\approx {\big\|}l_{g}|\big(\dot{F}_{p,q}(\mathbb{R}^{n},\{t_{k}\})%
\big)^{\ast }{\big\|}
\end{equation*}%
and every $l\in \big(\dot{F}_{p,q}(\mathbb{R}^{n},\{t_{k}\})\big)^{\ast }$
satisfies $l=l_{g}$ for some $g\in \dot{F}_{p^{\prime },q^{\prime }}(\mathbb{%
R}^{n},\{t_{k}^{-1}\})$.
\end{theorem}

\textbf{Acknowledgements. }This work was supported\ by the General Direction
of Higher Education and Training under\ Grant No. C00L03UN280120220004,
Algeria.


\begin{thebibliography}{99}
\bibitem{B03} O.V. Besov, Equivalent normings of spaces of functions of
variable smoothness, Function spaces, approximation, and differential
equations, A collection of papers dedicated to the 70th birthday of Oleg
Vladimorovich Besov, a corresponding member of the Russian Academy of
Sciences, Tr. May. Inst. Steklova, vol. 243, Nauka, Moscow 2003, pp. 87--95;
English transl. in Proc. Steklov Inst. Math. 243 (2003), 80--88.

\bibitem{B05} O.V. Besov, Interpolation, embedding, and extension of spaces
of functions of variable smoothness, Investigations in the theory of
functions and differential equations, A collection of papers dedicated to
the 100th birthday of academician Sergei Mikhailovich Nikol'skii, Tr. Mat.
Inst. Steklova, vol. 248, Nauka, Moscow 2005, pp. 52--63; English transl.
Proc. Steklov Inst. Math. 248 (2005), 47--58.

\bibitem{M07} M. Bownik, Anisotropic Triebel-Lizorkin spaces with doubling
measures, J. Geom. Anal. 17 (2007), 337--424.

\bibitem{M08} M. Bownik, Duality and interpolation of anisotropic
Triebel-Lizorkin spaces, Math. Z. 259 (2008), 131--169.

\bibitem{Bui82} H.Q. Bui, Weighted Besov and Triebel spaces: interpolation
by the real method, Hiroshima Math. J. 12 (1982), 581--605.

\bibitem{CF88} F. Cobos, D.L. Fernandez, Hardy-Sobolev spaces and Besov
spaces with a function parameter, In M. Cwikel and J. Peetre, editors,
Function spaces and applications, volume 1302 of Lect. Notes Math., pages
158--170. Proc. US-Swed. Seminar held in Lund, June, 1986, Springer, 1988.

\bibitem{D20} D. Drihem, Besov spaces\ with general weights, J. Math. Study,
56 (2023), 18--92.

\bibitem{D20.1} D. Drihem, Triebel-Lizorkin spaces with general weights,
Adv. Oper. Theory 8, 5 (2023). https://doi.org/10.1007/s43036-022-00230-0.

\bibitem{D23} D. Drihem, Complex interpolation of function spaces with
general weights, arXiv:2009.12223v3.

\bibitem{DT} O. Dominguez, S. Tikhonov. Function spaces of logarithmic
smoothness: embeddings and characterizations. To appear in Memoirs Amer.
Math. Soc. arxiv:1811.06399

\bibitem{ET96} D. Edmunds, H. Triebel, Spectral theory for isotropic fractal
drums. C.R. Acad. Sci. Paris 326, s\'{e}rie I (1998), 1269--1274.

\bibitem{ET99} D. Edmunds, H. Triebel, Eigenfrequencies of isotropic fractal
drums. Oper. Theory: Adv. \& Appl. 110 (1999), 81--102.

\bibitem{FL06} W. Farkas, H.-G. Leopold, Characterisations of function
spaces of generalised smoothness, Annali di Mat. Pura Appl. 185 (2006),
1--62.

\bibitem{FJ86} M. Frazier, B. Jawerth, Decomposition of Besov spaces,
Indiana Univ. Math. J. 34 (1985), 777--799.

\bibitem{FJ90} M. Frazier, B. Jawerth,\ A discrete transform and
decomposition of distribution spaces, J. Funct. Anal. 93 (1990), 34--170.

\bibitem{FrJaWe01} M. Frazier, B. Jawerth, G. Weiss, Littlewood-Paley Theory
and the Study of Function Spaces. CBMS Regional Conference Series in
Mathematics, vol. 79. Published for the Conference Board of the Mathematical
Sciences, Washington, DC. American Mathematical Society, Providence (1991).

\bibitem{GR85} J. Garc\.{\i}a-Cuerva, J. L. Rubio de Francia. Weighted Norm
Inequalities and Related Topics, North-Holland Mathematics Studies, 116.
Notas de Matem\'{a}tica [Mathematical Notes], 104. North-Holland Publishing
Co., Amsterdam, 1985.

\bibitem{Go79} M.L. Goldman, A description of the traces of some function
spaces, Trudy Mat. Inst. Steklov. 150 (1979), 99--127, (Russian) English
transl.: Proc. Steklov Inst. Math. 1981, no. 4 (150).

\bibitem{Go83} M.L. Goldman, A method of coverings for describing general
spaces of Besov type, Trudy Mat. Inst. Steklov. 156 (1980), 47--81,
(Russian) English transl.: Proc. Steklov Inst. Math. 1983, no. 2 (156).

\bibitem{L. Graf14} L. Grafakos, Classical Fourier Analysis, Third Edition,
Graduate Texts in Math., no 249, Springer, New York, 2014.

\bibitem{Ka83} G.A. Kalyabin, Description of functions from classes of
Besov-Lizorkin-Triebel type, Tr. Mat. Inst. Steklova. 156 (1980), 82--109,
(Russian) English translation: Proc. Steklov Inst. Math. 1983, (156).

\bibitem{Kl87} G.A. Kalyabin, P.I. Lizorkin, Spaces of functions of
generalized smoothness, Math. Nachr. 133 (1987), 7--32.

\bibitem{Ky03} G. Kyriazis. Decomposition systems for function spaces,
Studia Math. 157(2) (2003), 133--169.

\bibitem{Mo01} S.D. Moura, Function spaces of generalised smoothness, Diss.
Math. 398 (2001), 88 pp.

\bibitem{Mu72} B. Muckenhoupt, Weighted norm inequalities for the Hardy
maximal function, Trans. Amer. Math. Soc. 165 (1972), 207--226.

\bibitem{IzSa12} M. Izuki, Y. Sawano, Atomic decomposition for weighted
Besov and Triebel-Lizorkin spaces, Math. Nachr. 285 (2012), 103--126.

\bibitem{St93} E.M. Stein. Harmonic analysis: real-variable methods,
orthogonality, and oscillatory integrals, Princeton Mathematical Series, 43.
Monographs in Harmonic Analysis, III. Princeton University Press, Princeton,
NJ, 1993.

\bibitem{Tang} C. Tang, A note on weighted Besov-type and
Triebel-Lizorkin-type spaces, Journal of Function Spaces and Applications
Volume 2013, Article ID 865835, 12 pages.

\bibitem{T1} H. Triebel, Theory of Function Spaces, Birkh\"{a}user Verlag,
Basel, 1983.

\bibitem{T2} H. Triebel, Theory of Function Spaces II, Birkh\"{a}user
Verlag, Basel, 1992.

\bibitem{Ty15} A.I. Tyulenev, Some new function spaces of variable
smoothness, Sb. Math. 206 (2015), 849--891.

\bibitem{Ty-151} A.I. Tyulenev, On various approaches to Besov-type spaces
of variable smoothness, J. Math. Anal. Appl. 451 (2017), 371--392.

\bibitem{YY1} D. Yang, W. Yuan, A new class of function spaces connecting
Triebel-Lizorkin spaces and $Q$\ spaces, J. Funct. Anal. 255 (2008),
2760--2809.

\bibitem{YY2} D. Yang, W. Yuan, New Besov-type spaces and
Triebel-Lizorkin-type spaces including $Q$ spaces, Math. Z. 265 (2010),
451--480.
\end{thebibliography}
\end{document}